\documentclass[a4paper,english,11pt]{article}
\usepackage[utf8]{inputenc}
\usepackage{babel}
\usepackage{longtable}
\usepackage{psboxit}
\usepackage{vmargin,hhline,pifont}
\usepackage{enumerate,fancyhdr}
\usepackage{amsmath}
\usepackage{amsthm}
\usepackage{amsfonts}
\usepackage{amssymb}
\usepackage{chapterbib}
\usepackage{marvosym}
\usepackage[official]{eurosym}
\usepackage{verbatim}
\usepackage[arrow, matrix, curve]{xy}
\usepackage{makeidx}

\newcommand{\ad}{\mathrm{ad}}
\newcommand{\ndash}{\nobreakdash-\hspace{0pt}}

\newtheoremstyle{axel}
{}		
{}		
{\itshape}	
{}		
{\bfseries}	
{}		
{0pt\penalty1000}		
{}		
\newtheoremstyle{axelnl}
{}		
{}		
{\itshape}	
{}		
{\bfseries}	
{}		
{\newline}	
{}		
\theoremstyle{axelnl}
\newtheorem{theorem}{Theorem}[section]
\newtheorem{definition}{Definition}[section]
\newtheorem{cor}{Corollary}[section]
\newtheorem{remark}{Remark}[section]
\newtheorem{lemma}{Lemma}[section]

\newtheorem{example}{Example}[section]

\theoremstyle{axel}

\newtheorem{lemmas}[lemma]{Lemma}

\newtheorem{definitions}[definition]{Definition}

\makeindex

\makeatletter
\newcommand{\needs}[1]{\par\penalty-100\begingroup \dimen@\pagegoal\advance\dimen@-\pagetotal
\ifdim #1>\dimen@
 \pagebreak
\break
\fi
\endgroup}
\makeatother
\begin{document}
\title{Kac-Moody geometry}
\author{Walter Freyn}
\date{March 2010}
\maketitle

\begin{abstract}
The geometry of symmetric spaces \cite{Helgason01}, polar actions, isoparametric submanifolds \cite{Berndt03} and spherical buildings \cite{AbramenkoBrown08} is governed by spherical Weyl groups and simple Lie groups. The most natural generalization of semisimple Lie groups are affine Kac-Moody groups as they mirror their structure theory and have good explicitly known representations as groups of operators \cite{Tits84}, \cite{PressleySegal86}, \cite{Remy02}. In this article we describe the infinite dimensional differential geometry associated to affine Kac-Moody groups: Kac-Moody symmetric spaces \cite{Freyn07}, \cite{Freyn09}, isoparametric submanifolds in Hilbert space \cite{Terng89}, polar actions on Hilbert spaces \cite{terng95} and universal geometric twin buildings \cite{Freyn09}, \cite{freyn10b}.

\end{abstract}

\section{What is it all about}

In this article we describe the infinite dimensional differential geometry governed by affine Kac-Moody groups. 
The theory of Kac-Moody algebras and Kac-Moody groups emerged around 1960 independently in the works of V.\ G.\ Kac~\cite{Kac68}, R.\ V.\ Moody~\cite{Moody69}, I.\ L.\ Kantor~\cite{Kantor68} and D.-N.\ Verma (unpublished). From a formal, algebraic point of view, Kac-Moody algebras can be understood as realizations of generalized Cartan matrices~\cite{Kac90}. Hence Kac-Moody algebras appear as a natural generalization of simple Lie algebras. Furthermore it was pointed out that there is an explicit description of affine Kac-Moody algebras in terms of extensions of loop algebras. This point of view links the theory of affine Kac-Moody algebras to the theory of simple Lie algebras in another very elementary geometric way. A completion of the loop algebras with respect to various norms opens the way to the use of functional analytic methods~\cite{PressleySegal86}.

The next milestone was the discovery of a close link between Kac-Moody algebras and infinite dimensional differential geometry around 1990~\cite{HPTT}. Chuu-Lian Terng proved that certain isoparametric submanifolds in Hilbert spaces and polar actions on Hilbert spaces can be described using completions of algebraic Kac-Moody algebras~\cite{Terng89}, \cite{terng95}. In the parallel finite dimensional theory, the isometry groups are semisimple Lie groups. Hence this points again to the close similarity between semisimple Lie groups and Kac-Moody groups. During the same time Jacques Tits developed the idea of twin buildings and twin BN-pairs which are associated to algebraic Kac-Moody groups much the same way as buildings and BN-pairs are associated to simple Lie groups~\cite{Tits84}.  All those developments hint to the idea, that the rich subject of finite dimensional geometric structures whose symmetries are described by Lie groups should have an infinite dimensional counterpart, whose symmetries are affine Kac-Moody groups~\cite{Heintze06}. 

The most important among those objects are the following four classes:

\begin{samepage}
\begin{itemize}
\item[-] symmetric spaces,
\item[-] polar actions,
\item[-] buildings over the fields $\mathbb{R}$ or $\mathbb{C}$,
\item[-] isoparametric submanifolds.
\end{itemize}
\end{samepage}

In the first two examples, homogeneity is intrinsic, in the second two examples homogeneity is a priori an additional assumption. However, homogeneity can be proven under the assumption of sufficiently high rank.
The geometry of those examples reflects the two most important decompositions of the Lie groups:  the structure of symmetric spaces is a consequence of the Iwasawa decomposition, the structure of the building is a consequence of the system of parabolic subgroups and connected to the Bruhat decomposition. 
Let us note, that similar decompositions exist also for affine Kac-Moody groups.

In the investigation of those classes of objects one encounters two decisive structural elements: First, they all have a class of special \glqq flat\grqq\ subspaces (resp.\ subcomplexes) equipped with the action of a spherical reflection group. Subgroups of this reflection group fix special subspaces (resp.\ lower dimensional cells). Those reflection groups correspond to the Weyl groups of simple Lie algebras. Second, those flat subspaces are pieced together such that they meet in spheres around this set of special subspaces.

Evidently we want a similar feature for our infinite dimensional theory. Taking into account that Weyl groups of affine Kac-Moody algebras are affine Weyl groups, that is reflection groups on spaces of signature $(+++++++++0)$, we expect the reflection groups appearing in Kac-Moody geometry to be of this type. Unfortunately the metric $(+++++++++0)$ is degenerate. To get a nondegenerate metric, as we wish it for the construction of pseudo-Riemannian manifolds (and hence symmetric spaces), we have to add an additional dimension and choose a metric, that is defined by a pairing between the last coordinate and our new one. Hence the resulting spaces will be Lorentzian.

\noindent This philosophy gives us the following picture:

{\it There is a class of Lorentz symmetric spaces whose group of transvections is an affine Kac-Moody group.  The classification of irreducible Kac-Moody symmetric spaces is analogous to the one of irreducible finite dimensional symmetric spaces. Their isotropy representations induce polar actions on the Lie algebra. Because of the additional dimension we expect the essential part to be $1$-codimensional.
Principal orbits are isoparametric submanifolds. There is a class of buildings, whose chambers correspond to the points of the isoparametric submanifolds. As the system of parabolic subgroups consists of two \glqq opposite\grqq\ conjugacy classes, the building will consist of two parts; as the Bruhat decompositions do not exist on the whole Kac-Moody group, each of those parts is highly disconnected.}

By work done during the last 20 years this picture is now well established. In this survey we describe the main objects and connections between them, focusing on more recent work. 

  In section \ref{section:reflection_groups} we start by the investigation of spherical and affine Coxeter groups and describe the connections between them~\cite{AbramenkoBrown08}, \cite{Davis08}. Then we describe the finite dimensional theory (section \ref{section:the_finite_dimensional_blue_print})~\cite{Helgason01}, \cite{Berndt03}, \cite{AbramenkoBrown08}.
After that, we turn to the infinite dimensional theory and introduce affine Kac-Moody algebras (section \ref{section:geometric_affine_Kac-Moody_algebras}) and their Kac-Moody groups (section \ref{section:Affine_Kac-Moody_groups})~\cite{Kac90}, \cite{PressleySegal86}, \cite{Freyn09}. Having described the symmetry groups, we turn to the geometry governed by them: In section \ref{section:Kac-Moody_symmetric_spaces} we introduce Kac-Moody symmetric spaces~\cite{Freyn07}, \cite{Freyn09}. Then we turn to infinite dimensional polar actions (section \ref{section:Polar_actions_on_Hilbert_and_Frechet_spaces})~\cite{terng95}, and isoparametric submanifolds (section \ref{section:Isoparametric_submanifolds})~\cite{Terng89}. Finally, in section \ref{section:Universal_geometric_twin_buildings}, we investigate universal geometric twin buildings~\cite{Freyn09}. The last section, section \ref{section:conclusion_and_outlook}, is devoted to a description of open problems and directions for future research.

\section{Reflection groups}
\label{section:reflection_groups}

Let $V$ be a vector space equipped with a not necessarily positive definite inner product $\langle\ ,\ \rangle$.  An (affine) hyperplane $H$ in $V$ uniquely determines an (affine) reflection $s_H: V\longrightarrow V$.  A reflection group $W$ is a group of reflections, such that $W$ equipped with the discrete topology acts properly discontiuous on $V$. Hence it is defined by a set of hyperplanes $\mathcal{H}=\{H_i, i \in I\}$ such that $s_j(H_i)\subset \mathcal{H}\ \forall i,j$. Each reflection group is a discrete subgroup of the group $O(\langle\ , \rangle)$.

The structure of the resulting groups depends now on the metric $\langle\ ,\ \rangle$ on $V$.

The most common case is the one such that $\langle\ ,\ \rangle$ is positive definite, i.e.\ without loss of generality the standard Euclidean metric. The resulting reflection group is a discrete subgroup of the orthogonal group $O(\langle\ ,\ \rangle)$. As the unit sphere of $\langle\ ,\ \rangle$ is compact and $W $ is supposed to act properly discontinuous, we deduce that $W$ is a finite group. 

Define a group $W$ to be a Coxeter group if it admits a presentation of the following type:
$$W:=\langle s_1, \dots, s_n, n\in \mathbb{N} | s_i^2=1, (s_i s_k)^{m_{ik}}=1, i, k = 1, \dots, n \textrm{ and } m_{ik} \in \mathbb{N}\cup \infty\rangle\,.$$

Finite reflection groups are exactly finite Coxeter groups. Furthermore any finite Coxeter group admits a realization as a subgroup of the orthogonal group of a positive definite metric~\cite{Bourbaki02}.

$K:=\{1, \dots, n\}$ is called the indexing set. If $K'\subset K$, we use the notation $W_{K'}\subset W \cong W_K$ for the sub-Coxeter group generated by $K'$. The matrix $M=(m_{ik})_{i,k\in K}$ is called the Coxeter matrix. A Coxeter system is a pair $(W,S)$ consisting of a Coxeter group $W$ and a set of generators $S$ such that $\textrm{ord}(s)=2$ for all $s\in S$.

Any group element $w\in W$ is a word in the generators $s_i \in S$.  We define the length $l(w)$ of an element $w\in W$ to be the length of the shortest word representing $w$. The specific word -- and thus the length $l(w)$ -- depends on the specified set $S$ of generators. Nevertheless many global properties are preserved by a change of generators -- see~\cite{Davis08}.

\noindent Simple finite Coxeter groups have a complete classification:

\begin{enumerate}
 \item type $A_n$ is the symmetric group in $n$-letters.
 \item type $C_n$ resp.\ $B_n$ is the group of signed permutations of $n$  elements.
 \item type $D_n$ is the Weyl group of the orthogonal groups $SO(2n, \mathbb{C})$. 
 \item type $G_2, F_4, E_6, E_7, E_8$ are the Weyl groups of the Lie groups of the same names.
 \item type $H_3$ is the symmetry group of the 3-dimensional Dodecahedron and the Icosahedron, $H_4$ is the symmetry group of a regular $120$-sided solid in $4$ space whose $3$-dimensional faces are dodecahedral.
\item $I_n, n=5$ or $n>7$ is the dihedral group of order $2n$. We have the equivalences $I_1=\mathbb{Z}_2$, $I_2=A_1\times A_1$, $I_3=A_2$, $I_4=C_2$, $I_6=G_2$.
\end{enumerate}
 
Call a Coxeter group \glqq crystallographic\grqq\ if it stabilizes a lattice. The crystallographic Coxeter groups are  types $A_n, C_n, D_n, E_6, E_7, E_8, F_4, G_2$. Crystallographic Coxeter groups are exactly those that appear as Weyl groups of the root systems of finite simple Lie algebras~\cite{Bourbaki02}.

Closely related to finite Coxeter groups are affine Weyl groups. They are discrete subgroups of the group of affine transformations of a vector space. This group is a semidirect product of $O(\langle\ ,\ \rangle)$ with the group of translations of $V$. Hence $W_{\textrm{aff}}=L \rtimes W$ with $W$ a finite Coxeter group and $L\cong \mathbb{Z}^k$ a lattice in $V$. In particular $W$ is crystallographic. There is an easy way to linearize affine transformations by embedding the affine $n$\ndash space $V^n$ into an affine subspace i.e. the one defined by\ $x_{n+1}=1$ of an $n+1$\ndash dimensional vector space $V^{n+1}$. Then the group of affine transformations on  $V^n$ embeds into the general linear group $GL(V^{n+1})$~\cite{AbramenkoBrown08}.  
$V^{n+1}$ carries a degenerate metric of signature $(n,0)$. Hence we can interpret affine Weyl group as linear reflection groups in a vector space with a degenerate metric. In this form they appear as Weyl groups of affine Kac-Moody algebras. To get a non-degenerate metric on the Kac-Moody algebra we will add a further direction, to pair it with the degenerate direction. Hence the resulting $n+1$\ndash dimensional vector space will carry a Lorentz structure~\cite{Kac90}.

The next natural case are reflection groups in a space with a Lorentz metric. They appear as Weyl groups of hyperbolic Kac-Moody algebras \cite{Ray06}. While hyperbolic Kac-Moody algebras have applications in physics and in mathematics for example in M-theory and supergravity, there is up to now no differential geometry developed admitting hyperbolic Kac-Moody groups as symmetry groups.

\section{The finite dimensional blue print}
\label{section:the_finite_dimensional_blue_print}

This section introduces the finite dimensional geometry whose symmetries are governed by semisimple Lie groups. Useful references are  \cite{Helgason01}, \cite{Berndt03}, \cite{AbramenkoBrown08} and \cite{PalaisTerng88}.

\subsection{Semisimple Lie groups}

We define a Lie algebra to be semisimple if it has no Abelian ideals. It is simple if it is not $1$\ndash dimensional and has no nontrivial ideal. A Lie group is (semi\ndash )simple, if its Lie algebra is (semi\ndash )simple.

Classical examples are $SL(n, \mathbb{C}):=\{X\in \textrm{Mat}^{n\times n}(\mathbb{C})|\textrm{det}(X)=1\}$ with Lie algebra $\mathfrak{sl}(n, \mathbb{C}):=\{X\in \textrm{Mat}^{n\times n}(\mathbb{C})|\textrm{trace}(X)=0\}$, the orthogonal groups $SO(n, \mathbb{C}):=\{X\in \textrm{Mat}^{n\times n}(\mathbb{C})|XX^{t}=\textrm{Id}\}$. Every complex semisimple Lie algebra (resp.\ Lie group) is the direct product of complex simple Lie algebras (resp.\ groups). Hence we focus our study on simple Lie algebras and Lie groups.

We call a real simple Lie algebra $\mathfrak{g}$ a real form of the complex simple Lie algebra $\mathfrak{g}_{\mathbb{C}}$ if its complexification is isomorphic to $\mathfrak{g}_{\mathbb{C}}$. A simple Lie algebra or Lie group has various real forms:
Real forms of $SL(n,\mathbb{C})$ are among others $SU(n):=\{g\in SL(n, \mathbb{C})|g \overline{g}^t=\textrm{Id}\}$ and $SL(n, \mathbb{R}):=\{X\in \textrm{Mat}^{n\times n}(\mathbb{R)}|\textrm{det}(X)=1\}$

Real forms are in bijection with conjugate linear involutions: Fixing a real form, complex conjugation along this real form defines the involution. Starting with a conjugate linear involution, the fixed point set is a real form.

Among those real forms there is a unique distinguished compact one:

\begin{theorem}
Each complex simple Lie group has up to conjugation a unique compact real form. The same is true for simple Lie algebras.
\end{theorem}

\noindent For later reference we note an important decompositions theorem of simple Lie groups:

\begin{theorem}[Iwasawa decomposition]
 A simple Lie group $G$ has a decomposition $G=KAN$, where $K$ is maximal compact, $A$ is Abelian and $N$ is nilpotent.
\end{theorem}

Now we need to relate simple Lie algebras to spherical reflection groups:

\noindent Let $G$ be a compact simple Lie group. A torus $T$ is a maximal Abelian subgroup of $G$. In a matrix representation a torus is a subgroup of simultaneously diagonalizable elements --- a maximal torus is a maximal subgroup with this property. For example a torus in $SU(n)$ consists of all diagonal matrices $T:=\textrm{diag}(a_1, \dots, a_n| a_i\in \mathbb{C}, |a_i|=1, a_1\cdot .\ .\ . \cdot a_n=1)$. It is known that all maximal tori
are conjugate and that each element in $G$ is contained in at least one maximal torus. Choose an arbitrary torus $T$. The Weyl group is defined to be $W=N/T$ where $N$ is the normalizer of $T$. In the case of $SU(n)$ the Weyl group is the group of permutations of the $n$\ndash elements $a_1, \dots, a_n$ - hence it is the symmetric group in $n$ letters. The Weyl group is automatically a finite reflection group. Via the exponential map tori correspond to Abelian subalgebras.

For complex simple Lie groups a similar procedure is possible, but the situation is a little more complicated as not every element lies in a torus. We focus on the Lie algebra.  

\begin{definition}[Cartan subalgebra]
 A Cartan subalgebra of $\mathfrak{g}$ is a subalgebra $\mathfrak{h}$ of $\mathfrak{g}$ such that 
\begin{enumerate}
 \item $\mathfrak{h}$ is maximal Abelian in $\mathfrak{g}$
 \item For each $h\in \mathfrak{h}$ the endomorphism $\textrm{ad}(h): \mathfrak{g}\longrightarrow \mathfrak{g}$ is semisimple.
\end{enumerate}
\end{definition}

While all Cartan subalgebras are conjugate, it is no longer true, that they cover $\mathfrak{g}$.

Via the adjoint action of a Cartan subalgebra on the Lie algebra one defines the root system which is a refinement of the Weyl group action. The structure of this root system can be encoded into a matrix, called Cartan matrix.

\begin{definition}[Cartan matrix]

A Cartan matrix $A^{n\times n}$ is a square matrix with integer
coefficients such that
\begin{enumerate}
  \item $a_{ii}=2$ and $a_{i\not=j}\leq 0 \label{geometriccondition}$,
  \item $a_{ij}=0 \Leftrightarrow a_{ji}=0\label{liealgebracondition} $,
  \item There is a vector $v>0$ (component wise) such that $Av>0$ (component wise).
\end{enumerate}
\end{definition}

\begin{example}[$2\times 2$\ndash Cartan matrices]

There are -- up to equivalence -- four different $2$\ndash dimensional Cartan matrices:
$$\left(\begin{array}{cc} 2&0\\0&2  \end{array}\right),
\left(\begin{array}{cc} 2&-1\\-1&2  \end{array}\right),
\left(\begin{array}{cc} 2&-1\\-2&2  \end{array}\right),
\left(\begin{array}{cc} 2&-1\\-3&2  \end{array}\right).$$

\noindent They correspond to the Weyl groups of types
$A_1\times A_1, A_2, B_2, G_2$.
\end{example}

\begin{definition}
A Cartan matrix $A^{n\times n}$ is called decomposable iff $\{1, 2, \dots, n\}$ has a decomposition in two non-empty sets $N_1$ and $N_2$ such that $a_{ij}=0$ for $i\in N_1$ and $j\in N_2$. Otherwise it is called indecomposable.
\end{definition}

A complete list of indecomposable Cartan matrices consists of 
 $A_n$, $B_{n, n\geq 2}$, $C_{n, n\geq 3}$, $D_{n, n \geq 4}$, $E_6$, $E_7$, $E_8$, $F_4$, $G_2$~\cite{Bourbaki02}.

Conversely, starting with a Cartan matrix $A$ one can construct a Lie algebra $\mathfrak{g}(A)$ called its realization:

\begin{definition}[Realization]
Let $A^{n \times n}$ be a Cartan matrix. The realization of $A$, denoted $\mathfrak{g}(A)$, is the algebra
$$\mathfrak{g}(A^{n\times n}) = \langle e_i, f_i, h_i, i=1,\dots, n| \textrm{R}_1, \dots, \textrm{R}_6\rangle\,,$$ where
\begin{displaymath}
  \begin{array}{cl}
    \textrm{R}_1:& [h_i,h_j]=0\,,\\
    \textrm{R}_2:& [e_i,f_j]=h_i \delta_{ij}\,,\\
    \textrm{R}_3:& [h_i, e_j]=a_{ji}e_j\,,\\
    \textrm{R}_4:& [h_i, f_j]=-a_{ji}f_j\,,\\
    \textrm{R}_5:& (\ad e_i)^{1-a_{ji}}(e_j)=0 \hspace{3pt}(i\not= j)\,,\\
    \textrm{R}_6:& (\ad f_i)^{1-a_{ji}}(f_j)=0\hspace{3pt} (i\not= j)\,.\\
  \end{array}
\end{displaymath}
\end{definition}

In consequence there is a bijection between indecomposable Cartan matrices and complex simple Lie algebras; hence the classification of Cartan matrices yields also a complete list of simple Lie algebras. If a Cartan matrix $A^{(n+m)\times (n+m)}$ is decomposable into the direct sum of two Cartan matrices $A^{n\times n}$ and $A^{m\times m}$ then the same decomposition holds for the realizations: It is a direct product of (semi-)simple Lie algebras. This decomposition into direct factors is furthermore reflected in the structure of the geometric objects associated to those Lie algebras.

A further important structure property of simple Lie groups is the BN-pairs structure:

\begin{definition}[Borel subgroup, parabolic subgroup]
Let $G_{\mathbb{C}}$ be a complex simple Lie group.
A Borel subgroup $B$ is a maximal solvable subgroup. A subgroup $P \subset G_{\mathbb{C}}$ is called parabolic iff it contains a Borel subgroup. 
\end{definition} 

\begin{example} 
\begin{itemize}
 \item[-] The standard Borel subgroup in $SL(n,\mathbb{C})$ is the group of upper triangular matrices. All Borel subgroups are conjugate.
\item[-]  A standard parabolic subgroup in $SL(n, \mathbb{C})$ is an upper block-triangular matrix, that is an upper triangular matrix having blocks on its diagonal. There are several conjugacy classes of parabolic subgroups, corresponding to the various block-triangular matrices.
\end{itemize}
\end{example}

\noindent The BN-pair structure formalises the way those groups are assembled to yield a simple Lie group: 

\begin{definition}[BN\ndash pair]

Let $G_{\mathbb{C}}$ be a complex simple Lie group.  A set $(B,N,W,S)$ is a BN\ndash pair for $G$ iff:
\begin{enumerate}
\item $G=\langle B, N\rangle$. Moreover $T= B\cap N  \lhd N$ and $W=N/T$.
\item $s^2=1\hspace{3pt}\forall s\in  S$ and $W=\langle S\rangle$ and $(W,S)$ is a Coxeter system.
\item Let $C(w):=BwB$. Then $C(s)C(w)\subseteq C(w)\cup C(sw) \hspace{3pt} \forall s\in S$ and $w\in W$.
\item $\forall s\in S: sBs \not\subseteq B$.
\end{enumerate}
\end{definition}

\begin{theorem}[BN\ndash pairs and Bruhat decomposition]

Every complex simple Lie group $G$ has a unique BN\ndash pair structure.
Let $W$ be the Weyl group of $G$. Then 
$$G = \coprod_{w\in W} C(w)\,.$$
\end{theorem}

\begin{proof} See~\cite{Bump04}, section 30. \end{proof}

The Bruhat decomposition encodes the structure of the Tits buildings --- compare section \ref{subsection:spherical_buildings}.

\subsection{Symmetric spaces}

\begin{definition}
A (pseudo-)Riemannian symmetric space $M$ is a pseudo-Riemannian manifold $M$ such that for each $m\in M$ there is an isometry $\sigma_m:M\longrightarrow M$ such that $\sigma_m(m)=m$ and $d\sigma_m|_{T_mM} =-Id$.
\end{definition}

Direct consequences of the definition are that symmetric spaces are geodesically complete homogeneous spaces. Let $I(M)$ denote the isometry group of $M$ and $I(M)_m$ the isotropy subgroup of the point $m\in M$ then $M=I(M)/I(M)_m$. Let $g_m$ denote the metric on $T_mM$. Clearly $I(M)_m\subset O(g_m)$.
Hence for a Riemannian symmetric space, we find that $I(M)_m$ is a closed subgroup of a compact orthogonal group, hence compact.

We formalize those concepts

\begin{definition}[Symmetric pair]
Let $G$ be a connected Lie group, $H$ a closed subgroup. The pair $(G,H)$ is called a symmetric pair if there exists an involutive analytic automorphism $\sigma: G\longrightarrow G$ such that $\left(H_{\sigma}\right)_0 \subset H \subset H_{\sigma}$. Here $H_{\sigma}$ denotes the fixed points of $\sigma$ and $\left(H_{\sigma}\right)_0 $ its identity component. If $Ad_G(H)$ is compact, it is said to be Riemannian symmetric.
\end{definition}

Each symmetric space defines a symmetric pair. Conversely, each symmetric pair describes a symmetric space~\cite{Helgason01}.

\begin{definition}[OSLA]
An orthogonal symmetric Lie algebra is a pair $\mathfrak{g}, s$ such that 
\begin{enumerate}
 \item $\mathfrak{g}$ is a Lie algebra over $\mathbb{R}$,
\item $s$ is an involutive automorphism of $\mathfrak{g}$,
 \item the set of fixed points of $s$, denoted $\mathfrak{k}$, is a compactly embedded subalgebra. 
\end{enumerate}
\end{definition}

Clearly each Riemannian symmetric pair defines an OSLA. The converse is true up to coverings. 

Hence to give a classification of Riemannian symmetric spaces, we just have to classify OSLA's. 

We focus now our attention to the Riemannian case: The most important result is the following:
\begin{theorem}
 Let $M$ be an irreducible Riemannian symmetric space. Then either its isometry group is semisimple or $M=\mathbb{R}$.
\end{theorem}

In the non\ndash Riemannian case this is no longer true. While the pseudo\ndash Riemannian symmetric spaces with semisimple isometry group are completely classified~\cite{Berger57}, recent results of Ines Kath and Martin Olbricht~\cite{Kath04}, \cite{Kath06} show, that for pseudo-Riemannian symmetric spaces with a non\ndash semisimple isometry group, a classification needs a classification of solvable Lie algebras, which is out of reach.

There are two classes of irreducible Riemannian symmetric spaces with semisimple isometry group: Spaces of compact type and spaces of noncompact type.
\begin{definition}
 A Riemannian symmetric space is of compact type iff its isometry group is a compact semisimple Lie group. It is called of noncompact type if its isometry group is a noncompact semisimple Lie group. 
\end{definition}

\begin{theorem}
 Let $(\mathfrak{g}, s)$ be an orthogonal symmetric Lie algebra and $(L, U)$ a Riemannian symmetric pair associated to $(\mathfrak{g}, s)$:

\begin{enumerate}[i)]
\item If $(L,U)$ is of the compact type, then $L/U$ has sectional curvature $\geq 0$. 
\item If $(L,U)$ is of the noncompact type, then $L/U$ has sectional curvature $\leq 0$. 
\item If $(L,U)$ is of the Euclidean type, then $L/U$ has sectional curvature $= 0$.
\end{enumerate}
\end{theorem}

The Cartan-Hadamard theorem tells us that symmetric spaces of noncompact type are diffeomorphic to a vector space. Hence for every orthogonal symmetric Lie algebra of noncompact type, there is exactly one symmetric space of noncompact type. In contrast the topology of symmetric spaces of compact type is more complicated. As the fundamental group need not be trivial, for one orthogonal symmetric Lie algebra of the compact type there may be different symmetric spaces. There is always a simply connected one, which is the universal cover of all the others. 

Symmetric spaces of compact type and noncompact type appear in duality:
For every simply connected, irreducible symmetric space of the compact type, there is exactly one of the noncompact type and vice versa. 

\begin{example}
Take $(\mathfrak{l}, \mathfrak{u}):=(\mathfrak{so}{(n+1)}, \mathfrak{so}(n))$. 
A Riemannian symmetric pair associated to $(\mathfrak{l}, \mathfrak{u})$ is $(SO(n+1), SO(n))$. The corresponding symmetric space is isomorphic to the quotient $L/U$, hence is a sphere. Another symmetric space associated to $(\mathfrak{so}(n+1), \mathfrak{so}(n))$ is the projective space $\mathbb{R}P(n)$. 
The noncompact dual symmetric space is the hyperbolic space $\mathbb{H}^n=SO(n,1)/SO(n)$.
\end{example}

Besides $\mathbb{R}^n$, there are four classes of Riemannian symmetric spaces, two classes of spaces of compact type and two classes of spaces of noncompact type

\begin{enumerate}
 \item {\bf Type I} consists of coset spaces $G/K$, where $G$ is a compact simple Lie group and $K$ is a compact subgroup satisfying $Fix(\sigma)_0 \subseteq K\subseteq Fix(\sigma)$ for some involution $\sigma$. In this case $(L,U)=(G,K)$
\item {\bf Type II} consists of compact simple Lie groups $G$ equipped with their bi-invariant metric. $(L, U)=(G\times G, \Delta)$, where $\Delta=\{(x,x)\in G\times G\}$ is the diagonal subgroup.
\item {\bf Type III} consists of spaces $G/K$ where $G$ is a noncompact, real simple Lie group and $K$ a maximal compact subgroup. $(L,U)=(G,K)$.
 \item {\bf Type IV} consists of spaces $G_{\mathbb{C}}/G$ where $G_{\mathbb{C}}$ is a complex simple Lie group and $G$ a compact real form. $(L,U)=(G_{\mathbb{C}}, G)$.
\end{enumerate}

\noindent Types {\bf I} and {\bf III} are in duality as are types {\bf II} and {\bf IV}.

\subsection{Polar representations and isoparametric submanifolds}

\begin{definition}[Polar action]
Let $M$ be a Riemannian manifold.  An isometric action $G:M\longrightarrow M$ is called polar if  there exists a complete, embedded, closed submanifold $\Sigma \subset M$, that meets  each orbit orthogonally.
\end{definition}

\begin{definition}[Polar representation]
 A polar representation is a polar action on an Euclidean vector space, acting by linear transformations. 
\end{definition}

\begin{example}[Adjoint representation]
Let $G$ be a compact simple Lie group, $\mathfrak{g}$ its Lie algebra. The adjoint representation is polar. Sections are the Cartan subalgebras.
\end{example}

\begin{example}[s-representation]
 Let $(G, K)$ be a Riemannian symmetric pair, $M=G/K$ the corresponding symmetric space and $m=eK\in M$ (hence $K=I(M)_m$). The action of $K$ on $M$ induces an action on $T_mM$. Let $\mathfrak{g}=\mathfrak{k}\oplus \mathfrak{p}$ be the corresponding Cartan decomposition (i.e.\ $\mathfrak{k}=Lie(K)$. Then $\mathfrak{p}\cong T_mM$.  Using this isomorphism, we get an action of $K$ on $\mathfrak{p}$. This action, called the $s$\ndash representation of $M$, is polar. Each Abelian subalgebra in $\mathfrak{p}$ is a section.
\end{example}

\begin{theorem}[Dadok]
 Every polar representation on $\mathbb{R}^n$ is orbit equivalent to an $s$-representation.
\end{theorem}

\noindent The orbits of polar actions have an interesting geometric structure: We define

\begin{definition}[Isoparametric submanifold]
 A submanifold $S\subset V$ is called isoparametric if $S$ has a flat normal bundle and constant principal curvatures.
\end{definition}

\begin{theorem}
Principal orbits of a polar representation is an isoparametric submanifold.
\end{theorem}

This result can be proven by a direct verification. Conversely we have:

\begin{theorem}[Thorbergsson]
Each full irreducible isoparametric submanifold of $\mathbb{R}^n$ of rank at least three is an orbit of an $s$-representation.
\end{theorem}

The original proof by Gudlaugur Thorbergsson \cite{Thorbergsson91} proceeds by associating a Tits building to any isoparametric submanifold, relying on the fact that spherical buildings of rank at least three are classified~\cite{Thorbergsson91}. A second proof of Carlos Olmos proceeds by showing homogeneity~\cite{Berndt03}.

Hence most isoparametric submanifolds are homogeneous. Nevertheless, there are some examples in codimension $2$ which are non-homogeneous. A complete classification is still missing.

\subsection{Spherical buildings}
\label{subsection:spherical_buildings}

\subsubsection{Foundations}
There are several equivalent definitions of a building~\cite{AbramenkoBrown08}. For us the most convenient one is the $W$\ndash metric one:

\begin{definition}[Building - $W$\ndash metric definition]
\label{wmetricbuilding}
Let $(W,S)$ be a Coxeter system. A building of type $(W,S)$ is  a pair $(\mathcal{C}, \delta)$  consisting of a nonempty set $\mathcal{C}$ whose elements are called chambers together with a map $\delta: \mathcal{C}\times \mathcal{C}\longrightarrow W$, called the Weyl distance function, such that for all $C,D\in \mathcal{C}$ the following conditions hold: 
\begin{enumerate}
	\item $\delta(C, D)=1$ iff $C=D$.
  \item If $\delta(C,D)= w $ and $C'\in \mathcal{C}$ satisfies $\delta(C', C)=s\in S$ then $\delta(C',D)=sw$ or $w$. If in addition $l(sw)=l(w)+1$ then $\delta(C', D)=sw$.
	\item If $\delta(C,D)=w$ then for any $s\in S$ there is a chamber $C' \subset \mathcal{C}$, such that $\delta(C', C)=s$ and $\delta(C', D)=sw$.
\end{enumerate}
\end{definition}

\begin{example}
Let $G_{\mathbb{C}}$ be a complex simple Lie group. Set $\mathcal{C}=G_{\mathbb{C}}/B$ and define $\delta(f,g)=w$ iff $f^{-1}g\subset BwB$ in the Bruhat decomposition. Then $(\mathcal{C}, \delta)$ is a spherical building. 
\end{example}

\noindent Hence chambers are in bijection to Borel subgroups $B$.
This definition is equivalent to the standard definition of a building as a simplicial complex ---  for a proof see~\cite{AbramenkoBrown08}:

\begin{definition}[Building - chamber complex definition]
\label{building}
A building $\mathfrak{B}$ is a thick chamber complex  $\Sigma$ together with a set $\mathcal{A}$ of thin chamber complexes $A\in \mathcal{A}$, called apartments, satisfying the following axioms:
\begin{enumerate}
	\item For every pair of simplices $x,y \in \Sigma$ there is an apartment $A_{x,y}\subset \mathcal{A}$ containing both of them.
	\item Let $A$ and $A'$ be apartments, $x$ a simplex and $C$ a chamber such that $\{x, C\}\subset A\cap A'$. Then there is a chamber complex isomorphism $\varphi: A \longrightarrow A'$ fixing $x$ and $C$ pointwise.
\end{enumerate}
\end{definition}

In this definition the finite reflection groups appear via their action on the apartments: 

\noindent To each Coxeter System $(W,S)$ one can associate a simplicial complex $C$ of dimension $|S|-1$, called the Coxeter complex of type $(W,S)$.

 We start with $(W,S)$: Let $\mathcal{S}$ denote the power set of $S$ and define a partial order relation on $\mathcal{S}$ by $ S' < S'' \in \mathcal{S}$ iff $ (S')^c \subset (S'')^c$ as subsets of $S$. Here $(S')^c$ denotes the complement of $S'$ in $S$. Now construct a simplicial complex $\Sigma(S)$ associated to $\mathcal{S}$ by identifying a set $S'\subsetneq \mathcal{S}$ with a simplex $\sigma(S')$ of dimension $|S'^c|-1$ and defining the boundary relations of $\Sigma(S)$ via the partial order on $\mathcal{S}$: $\sigma(S')$ is in the boundary of $\sigma(S'')$ iff $S' <S''$. 
In this simplicial complex $\emptyset= (S)^c$ corresponds to a simplex of maximal dimension and the $|S|$ sets consisting of the single elements $s_i \in S$ correspond to faces.

The simplicial complex $\Sigma(W,S)$ consists of all $W$\ndash translates of elements in $\Sigma(S)$. Its elements $\sigma(w, S')$ correspond to pairs consisting of an element $w\in W$ and an element $S'\subset \mathcal{S}$ subject to the equivalent relation $\sigma(w_1, S'_1)\simeq \sigma(w_1, S'_1)$ iff $S'_1=S'_2$ and $w_1\cdot \langle S'_1\rangle= w_2 \cdot \langle S'_2 \rangle$. It carries a natural $W$\ndash action that is transitive on simplices of maximal dimension (called chambers). For a simplex $\sigma (w, \{s_i, i\in S'\})$, the stabilizer subgroup is $wW_{S'}w^{-1}$ as elements in $W_{S'}$ stabilize $\{s_i, i\in S'\}$.

As a simplicial complex the Coxeter complex is independent of the choice of $S$. 

The action of elements $s_i \in S$ on $\Sigma(W,S)$ can be interpreted geometrically as reflections at the faces $\sigma(e, s_i)$ of $\sigma(e, \emptyset)$, where $e$ denotes the identity element of $W$. For further details we refer to~\cite{Davis08}.

A chamber complex is a simplicial complex satisfying two taming properties: first it is required that every simplex is contained in the boundary of a simplex of maximal dimension  and second that for every pair of simplices $x$ and $y$ there exists a sequence of simplices of maximal dimension $S:=\{z_1, \dots, z_n\}$ such that $x\in z_1$, $y\in z_n$ and $z_i \cap z_{i+1}$ contains a codimension $1$ simplex. Simplices of maximal dimension are called chambers; simplices of codimension $1$ are called walls. A sequence $S$ is called a gallery connecting $x$ and $y$.

\noindent A chamber complex will be called thin if every wall is a face of exactly $2$ chambers. It will be called thick if every  wall is a face of at least $3$ chambers.

\noindent A chamber complex map $\varphi: A\longrightarrow A'$ is a map of simplicial complexes, mapping $k$\ndash simplices onto $k$\ndash simplices and respecting the face relation. It is a chamber complex isomorphism iff it is bijective.

\subsubsection{Buildings and polar actions}

Let $G$ be a real simple Lie group of compact type and $\mathfrak{g}$ its Lie algebra.
To understand how buildings fit into our picture, we define a $G$\ndash equivariant embedding into the Lie algebra. 

\noindent Start with the adjoint action:

$$\varphi: G \times \mathfrak{g} \longrightarrow \mathfrak{g}, \hspace{15pt}(g,X) \mapsto gXg^{-1}\, . $$
We want first to restrict the domain of definition to a fundamental domain for this action.

\noindent As $\mathfrak{g}$ is covered by conjugate maximal Abelian subalgebras $\mathfrak{t}$, the map
$$\varphi: G \times \mathfrak{t} \longrightarrow \mathfrak{g}, \hspace{15pt}(g,X) \mapsto gXg^{-1}$$
is surjective. As $T:=\exp{\mathfrak{t}}$ acts trivially on $\mathfrak{t}$, 
$$\varphi: G/T \times \mathfrak{t} \longrightarrow \mathfrak{g}, \hspace{15pt}(gT, X) \mapsto gXg^{-1}$$ is well defined and surjective. 
Let $\Delta\subset \mathfrak{t}$ be a fundamental domain for the action of the Weyl group $W:=N(T)/T \subset G/T$ on $\mathfrak{t}$. Then the map
$$\varphi: G/T \times \Delta_G \longrightarrow \mathfrak{g}, \hspace{15pt} (gT,X) \mapsto gXg^{-1}$$
 is again surjective.

For regular elements, i.e.\ all those in the interior $\Delta$, this map is injective~\cite{BroeckertomDieck85}. For all elements $Y$ in the boundary, the stabilizer subgroups are generated by the reflections $s_i\subset W$ fixing $Y$.  For an element $X \subset \Delta_{\mathfrak{g}}$ in the intersection of the faces $\Delta_{i_1}, \dots, \Delta_{i_k}, i\in I$  the stabilizer is $W_I=\langle S_I\rangle$.

As the adjoint action preserves the Cartan-Killing form $B(X,Y)$, it is an isometry with respect to the induced metric $\langle X, Y\rangle = - B(X,Y)$. Hence it preserves every sphere $S^n_R\subset \mathfrak{g}$ of fixed radius $R$. Thus we can define an $\textrm{Ad}$-invariant subspace $\mathfrak{g}_R:= \mathfrak{g}\cap S^{m}_R\subset \mathfrak{g}$. A fundamental domain for the Adjoint action on $\mathfrak{g}_R$ is the set $\Delta_{(\mathfrak{g},R)}:= \Delta_{\mathfrak{g}}\cap S^n_R$.

\noindent Accordingly we get a surjective map:
$$ \textrm{Ad}(G/T): \Delta_{(\mathfrak{g},R)}\mapsto \mathfrak{g}_R, \hspace{15pt} (g,X) \mapsto g X g^{-1}.$$

Now construct a simplicial complex like that:
For each element $X\in \Delta_{(\mathfrak{g}, R)}$ there is a subgroup $W_X\subset W$ stabilizing $X$. If $X$ is in the interior, then $W_X=\{e\}$. If it is in the boundary, $W_X$ is the group $W_I\subset W$ generated by the reflections $s_i$ that fix $X$. Thus we can replace $\Delta_{(\mathfrak{g}, R)}$ by the complex $\Sigma(S)$. It has exactly $n:=\textrm{rank } \mathfrak{g}$ faces $f_1, \dots, f_n$. Let $S=\{s_i\}, i = 1,\dots, n$ where $s_i$ denotes the reflection at $f_i$. $(W,S)$ is a Coxeter system. The cells in $\Delta_{(\mathfrak{g}, R)}$ correspond bijectively to subsets $S' \subset S$.

 The $W$\ndash translates $\Sigma(S)$  tessellate  $\mathfrak{t}_R:=\mathfrak{t}\cap S^n_R$; thus the tessellation of $\mathfrak{t}_R$ corresponds to the thin Coxeter complex $\Sigma(W,S)$. The $G/T$\ndash translates tessellate the whole sphere $\mathfrak{g}_R$. We get a simplicial complex

$$\mathfrak{B}_G:= (G/T \times \Delta)/\sim\,,$$ 

\noindent By the Iwasawa decomposition we get $B:= T \times A \times N$ and $G/T:= G_{\mathbb C}/B$. Hence:

$$\mathfrak{B}_G:= (G_{\mathbb C}/B \times \Delta)/\sim $$

Our simplicial complex is the building for $G$ and we get an embedding of the building $\mathfrak{B}_G:= (G/T \times \Delta)/\sim$ into the sphere $S_R$.

Moreover let $h \in G$ act on the building $\mathfrak{B}_G$ by left multiplication. Then the following diagram commutes:
\vspace{-1ex}
\begin{displaymath}
\begin{xy}
  \xymatrix{
      \mathfrak{B}_G \ar[r]^h \ar[d]_i    &   \mathfrak{B}_G  \ar[d]^i  \\
      \mathfrak{g}_R \ar[r]^{\textrm{Ad}(h)}             &   \mathfrak{g}_R  
  }
\end{xy}
\end{displaymath}

\noindent The description $\mathfrak{B}_G:= (G_{\mathbb C}/B \times \overline{\Delta})/\sim $ shows that $G_{\mathbb{C}}$ acts from the left on the  building $\mathfrak{B}_G= (G_{\mathbb C}/B \times \Delta)/\sim$.

\noindent In this complex description chambers correspond exactly to Borel subgroups, hence we find:

\begin{theorem}
Let $G_{\mathbb{C}}$ be a simple Lie group of type $X_l$ and $\mathfrak{B}$ the building of the same type.
\begin{enumerate}
\item The chambers of $\mathfrak{B}$ correspond to the Borel subgroups of $G_{\mathbb{C}}$, simplices of $\mathfrak{B}$ correspond to parabolic subgroups. The correspondence can be realized by associating to every simplex $c\in \mathfrak{B}$ its stabilizer subgroup $P_c$ in $G_{\mathbb{C}}$ with respect to the left action, described above. 

\item A simplex $c$ is in the boundary of a cell $d$ iff $P_d\subset P_c$.  
\end{enumerate}
\end{theorem}

\begin{remark}
It is an important observation that in a simple complex Lie group all Borel subgroups are conjugate. In building theoretical language this translates to the fact that the building is a connected simplicial complex.
\end{remark}

The association of buildings to symmetric spaces can be done in the same way as we did it for Lie groups, with the adjoint representation replaced by the isotropy representation~\cite{Helgason01}, \cite{Mitchell88}, \cite{Eberlein96}). A good survey for applications of buildings in finite dimensional geometry is \cite{Ji06}.

\section{Kac-Moody groups and their Lie algebras}

\subsection{Geometric affine Kac-Moody algebras}
\label{section:geometric_affine_Kac-Moody_algebras}

The classical references for (algebraic) Kac-Moody algebras are the books \cite{Kac90} and \cite{Moody95}. We will encounter algebraic Kac-Moody algebras and various analytic completions. In some cases of ambiguity we use the denomination algebraic Kac-Moody algebra for the classical Kac-Moody algebras.

\begin{definition}[affine Cartan matrix]
An affine Cartan matrix $A^{n\times n}$ is a square matrix with integer
coefficients, such that
\begin{enumerate}
  \item $a_{ii}=2$ and $a_{i\not=j}\leq 0 \label{geometricconditionaffine}$.
  \item $a_{ij}=0 \Leftrightarrow a_{ji}=0\label{liealgebraconditionaffine}$.
  \item There is a vector $v>0$ (component wise) such that $Av=0$.
\end{enumerate}
\end{definition}

\begin{example}[$2\times 2$\ndash affine Cartan matrices]
There are -- up to equivalence -- two different $2$\ndash dimensional affine Cartan matrices:
$$\left(\begin{array}{cc} 2&-2\\-2&2  \end{array}\right),
\left(\begin{array}{cc} 2&-1\\-4&2  \end{array}\right)\,.$$

\noindent They correspond to the non-twisted algebra
$\tilde{A}_1$ and the twisted algebra $\tilde{A}_1'$.

\end{example}

\begin{enumerate}
\item The indecomposable non-twisted affine Cartan matrices are
$$\widetilde{A}_n, \widetilde{B}_n, \widetilde{C}_n, \widetilde{D}_n, \widetilde{E}_6, \widetilde{E}_7, \widetilde{E}_8, \widetilde{F}_4, \widetilde{G}_2\,.$$
Every non-twisted affine Cartan matrix $\widetilde{X}_l$  can be constructed from a (finite) Cartan matrix $X_l$ by the addition of a further line and column. The denomination as ``non-twisted'' points to the explicit construction as loop algebras.
\item The indecomposable twisted affine Cartan matrices are
$$ \widetilde{A}_1', \widetilde{C}_l', \widetilde{B}_l^t, \widetilde{C}_l^t, \widetilde{F}_4^t, \widetilde{G}_2^t\,.$$
The Kac-Moody algebras associated to them can be constructed as fixed point algebras of certain automorphisms $\sigma$ of a non-twisted Kac-Moody algebra $X$. This construction suggests an alternative notation describing a twisted Kac-Moody algebra by the order of $\sigma$ and the type of $X$.
This yields the following equivalences:

\begin{displaymath}
  \begin{array}{ccl}
  \widetilde{A}_1'   & \hspace{30pt} & ^2\widetilde{A}_{2}\\
  \widetilde{C}_l'   & \hspace{30pt} & ^2\widetilde{A}_{2l}, l\geq 2\\
  \widetilde{B}_l^t  & \hspace{30pt} & ^2\widetilde{A}_{2l-1}, l\geq 3\\
  \widetilde{C}_l^t  & \hspace{30pt} & ^2\widetilde{D}_{l+1}, l\geq 2\\
  \widetilde{F}_4^t  & \hspace{30pt} & ^2\widetilde{E}_6\\
  \widetilde{G}_2^t  & \hspace{30pt} & ^3\widetilde{D}_4\\ 
 \end{array}
\end{displaymath}
\end{enumerate}

\noindent As in the finite dimensional case to every affine Cartan matrices $A$ one can associate a realizations $\mathfrak{g}(A)$. Those correspond exactly to the affine Kac-Moody algebras. Fortunately besides this abstract approach using generators and relations there is a very concrete second description for affine Kac-Moody algebras, namely the loop algebra approach. To describe the loop algebra approach to Kac-Moody algebras we follow the terminology of the article~\cite{Heintze09}.

 Let $\mathfrak{g}$ be a finite dimensional reductive Lie algebra over the field $\mathbb{F}=\mathbb{R}$ or $\mathbb{C}$. Hence $\mathfrak{g}$ is a direct product of a semisimple Lie algebra $\mathfrak{g}_s$ with an Abelian Lie algebra $\mathfrak{g}_a$. Let furthermore $\sigma \in \textrm{Aut}(\mathfrak{g}_s)$ denote an automorphism of finite order of $\mathfrak{g}_s$ such that  $\sigma|_{\mathfrak{g}_a}=Id$. If $\mathfrak{g}_s$ is a Lie algebra over $\mathbb{R}$ we suppose it to be of compact type.
$$L(\mathfrak g, \sigma):=\{f:\mathbb{R}\longrightarrow \mathfrak{g}\hspace{3pt}|f(t+2\pi)=\sigma f(t), f \textrm{ satisfies some regularity condition}\}\label{abstractkacmoodyalgebra}\,.$$
We use the notation $L(\mathfrak g, \sigma)$ to describe in a unified way algebraic constructions that apply to various explicit realizations of loop algebras satisfying sundry regularity conditions --- i.e.\ smooth, real analytic, (after complexification) holomorphic or algebraic loops. If we discuss loop algebras of a fixed regularity we use other precise notations: $M\mathfrak{g}$ for holomorphic loops on $\mathbb{C}^*$, $L_{alg}\mathfrak{g}$ for algebraic, $A_n\mathfrak{g}$ for holomorphic loops on the annulus $A_n=\{z\in \mathbb{C}|e^{-n}\leq |z|\leq e^n\}$.

\begin{definition}[Geometric affine Kac-Moody algebra]
\label{geometricaffinekacmoodyalgebra}
The geometric affine Kac-Moody algebra associated to a pair $(\mathfrak{g}, \sigma)$ is the algebra:
$$\widehat{L}(\mathfrak g, \sigma):=L(\mathfrak g, \sigma) \oplus \mathbb{F}c \oplus \mathbb{F}d\,,$$
equipped with the lie bracket defined by:
\begin{alignat*}{1}
  [d,f]&:=f';[c,c]=[c,d]=[c,f]=[d,d]=0\,;\\
  [f,g]&:=[f,g]_{0} + \omega(f,g)c\,.
\end{alignat*}

\noindent Here $f\in L(\mathfrak g, \sigma)$ and $\omega$ is a certain antisymmetric $2$\ndash form on $M\mathfrak g$, satisfying the cocycle condition.
\end{definition}

Let us remark that in contrast to the usual Kac-Moody theory, $\mathfrak{g}$ has not to be simple, but may be reductive, i.e.\ a product of semisimple Lie algebra with an Abelian one. The algebra $\widetilde{L}(G, \sigma):=L(\mathfrak g, \sigma) \oplus \mathbb{F}c$ is called the derived algebra.

We give some further definitions:

\begin{definition}
A real form of a complex geometric affine Kac-Moody algebra $\widehat{L}(\mathfrak{g}_{\mathbb{C}}, \sigma)$ is the fixed point set of a conjugate linear involution. 
\end{definition}

Involutions of a geometric affine Kac-Moody algebra restrict to involutions of irreducible factors of the loop algebra. Hence the invariant subalgebras are direct products of invariant subalgebras in those factors together with the appropriate torus extension.

\begin{definition}[compact real affine Kac-Moody algebra]
A compact real form of a complex affine Kac-Moody algebra $\widehat{L}(\mathfrak{g}_{\mathbb{C}}, \sigma)$ is  a real form which is isomorphic to $\widehat{L}(\mathfrak{g}_{\mathbb {R}}, \sigma)$
where $\mathfrak{g}_{\mathbb R}$ is a compact real form of
$\mathfrak{g}_{\mathbb C}$. 
\end{definition}

\begin{remark}
A semisimple Lie algebra is called of ``compact type'' iff it integrates to a compact semisimple Lie group.
The infinite dimensional generalization of compact Lie groups are loop groups of compact Lie groups and their Kac-Moody groups, constructed as extensions of those loop groups (see section~\ref{section:Affine_Kac-Moody_groups}). Thus the denomination is justified by the fact that ``compact'' affine Kac-Moody algebras integrate to ``compact'' Kac-Moody groups.
\end{remark}

To define a loop group of the compact type we use an infinite dimensional version of the Cartan-Killing form:

\begin{definition}[Cartan-Killing form]
The Cartan-Killing form of a loop algebra $L(\mathfrak{g}_{\mathbb{C}},\sigma)$ is defined by
$$B_{(\mathfrak{g}_{\mathbb{C}},\sigma)}\left(f, g\right)=\int_{0}^{2\pi} B\left(f(t), g(t) \right)dt\,.$$ 
\end{definition}

\begin{definition}[compact loop algebra]
A loop algebra of compact type is a subalgebra of $L(\mathfrak{g}_{\mathbb{C}},\sigma)$ such that its Cartan-Killing form is negative definite. 
\end{definition}

\begin{lemma}
Let $\mathfrak{g}_{\mathbb{R}}$ be a compact semisimple Lie algebra.
Then the loop algebra $L(\mathfrak{g}_{\mathbb{R}}, \sigma)$ is of compact type.
\end{lemma}

\begin{proof}
The Cartan-Killing form on $\mathfrak{g}_{\mathbb{R}}$ is negative definite. Hence $B_{(\mathfrak{g}_{\mathbb{C}},\sigma)}\left(f, g\right)$ is negative definite.
\end{proof}

To find noncompact real forms we need the following result of Ernst Heintze and Christian Groß (Corollary 7.7.\ of~\cite{Heintze09}):

\begin{theorem}
Let $\mathcal{G}$ be an irreducible complex geometric affine Kac-Moody algebra i.e.\ $\mathcal{G}= \widehat{L}(\mathfrak{g}, \sigma)$ with $\mathfrak{g}$ simple, $\mathcal{U}$ a real form of compact type. The conjugacy classes of real forms of noncompact type of $\mathcal{G}$ are in bijection with the conjugacy classes of involutions on $\mathcal{U}$. The correspondence is given by $\mathcal{U}=\mathcal{K}\oplus \mathcal{P}\mapsto \mathcal{K}\oplus i\mathcal{P}$ where $\mathcal{K}$ and $\mathcal{P}$ are the $\pm 1$\ndash eigenspaces of the involution.  
\end{theorem}

Thus to find noncompact real forms we have to study automorphism of order $2$ of a geometric affine Kac-Moody algebra of the compact type. From now on we restrict to involutions $\widehat{\varphi}$ of type $2$, that is $\widehat{\varphi}(c)=-c$.

A careful examination of the construction of a geometric affine Kac-Moody algebra of a non simple Lie algebra allows, to extend this result to the broader class of geometric affine Kac-Moody algebras~\cite{Freyn09}:

\begin{theorem}
Let $\mathcal{G}$ be a complex geometric affine Kac-Moody algebra, $\mathcal{U}$ a real form of compact type. The conjugacy classes of real forms of noncompact type of $\mathcal{G}$ are in bijection with the conjugacy classes of involutions on $\mathcal{U}$. The correspondence is given by $\mathcal{U}=\mathcal{K}\oplus \mathcal{P}\mapsto \mathcal{K}\oplus i\mathcal{P}$ where $\mathcal{K}$ and $\mathcal{P}$ are the $\pm 1$\ndash eigenspaces of the involution. Furthermore every real form is either of compact type or of noncompact type. A mixed type is not possible.
\end{theorem}

\begin{lemma}
Let $\mathfrak{g}$ be semisimple and $\widehat{L}(\mathfrak{g}, \sigma)_D$ be a real form of the noncompact type. Let $\widehat{L}(\mathfrak{g}, \sigma)_D=\mathcal{K}\oplus \mathcal{P}$ be a Cartan decomposition. The Cartan Killing form is negative definite on $\mathcal{K}$ and positive definite on $\mathcal{P}$
\end{lemma}

\begin{proof}
Suppose first $\sigma$ is the identity. Let $\varphi$ be an automorphism. Then without loss of generality $\varphi(f)=\varphi_0(f(-t))$~\cite{Heintze09}. Let $\mathfrak{g}=\mathfrak{k}\oplus \mathfrak{p}$ be the decomposition of $\mathfrak{g}$ into the $\pm 1$-eigenspaces of $\varphi_0$. Then
$f\in \textrm{Fix}(\varphi)$ iff its Taylor expansion satisfies
$$\sum_n a_n e^{int}= \sum \varphi_0(a_{-n})e^{int}\,.$$
Let $a_n=k_n\oplus p_n$ be the decomposition of $a_n$ into the $\pm 1$ eigenspaces with respect to $\varphi_0$. 
Hence $$f(t)=\sum_n k_n \cos(nt)+\sum_n p_n \sin(nt)\,.$$ 
Then using bilinearity and the fact that $\{\cos(nt), \sin(nt)\}$ are orthonormal we can calculate $B_{\mathfrak{g}}$:
$$B_\mathfrak{g}=\int_0^{2\pi}\sum_n \cos^{2}(nt)B(k_n, k_n)-\int_0^{2\pi}\sum_n \sin^{2}(nt)B(p_n, p_n)\,.$$
Hence $B_{\mathfrak{g}}$ is negative definite on $\textrm{Fix}(\varphi)$. Analogously one calculates that it is positive definite on the $-1$\ndash eigenspace of $\varphi$. 
If $\sigma\not=Id$ then one gets the same result by embedding $L(\mathfrak{g}, \sigma)$ into an algebra $L(\mathfrak{h}, \textrm{id})$ which is always possible~\cite{Kac90}. 

\end{proof}

\begin{lemma}
Let $\mathfrak{g}$ be Abelian. The Cartan-Killing form of $L(\mathfrak{g})$ is trivial.
\end{lemma}

\begin{proof}
Direct calculation.
\end{proof}

\noindent Now we can define OSAKAs, the Kac-Moody analogue of orthogonal symmetric Lie algebras \cite{Freyn09}:

\begin{definition}[Orthogonal symmetric Kac-Moody algebra]

An orthogonal symmetric affine Kac-Moody algebra (OSAKA) is a pair $\left(\widehat{L}(\mathfrak{g}, \sigma), \widehat{\rho}\right)$ such that
\begin{enumerate}
	 \item $\widehat{L}(\mathfrak{g}, \sigma)$ is a real form of an affine geometric Kac-Moody algebra,
	 \item $\widehat{\rho}$ is an involutive automorphism of $\widehat{L}(\mathfrak{g}, \sigma)$ of the second kind,
	 \item $\textrm{Fix}(\widehat{\rho})$ is a compact real form.
\end{enumerate}
\end{definition}

\noindent Following Helgason, we define $3$ types of OSAKAs:

\begin{definition}[Types of OSAKAs]

\label{types of OSAKAs}
Let $\left(\widehat{L}(\mathfrak{g}, \sigma),\widehat{\rho}\right)$ be an OSAKA. Let $\widehat{L}(\mathfrak{g}, \sigma)=\mathcal{K}\oplus \mathcal{P}$ be the decomposition of $\widehat{L}(\mathfrak{g}, \sigma)$ into the eigenspaces of $\widehat{L}(\mathfrak{g}, \sigma)$ of eigenvalue $+1$ resp.\ $-1$.
\begin{enumerate}
	\item If $\widehat{L}(\mathfrak{g}, \sigma)$ is a compact real affine Kac-Moody algebra, it is said to be of the compact type.
	\item If $\widehat{L}(\mathfrak{g}, \sigma)$ is a noncompact real affine Kac-Moody algebra, $\widehat{L}(\mathfrak{g}, \sigma)=\mathcal{U}\oplus \mathcal{P}$ is a Cartan decomposition of $\widehat{L}(\mathfrak{g}, \sigma)$.
	\item If $L(\mathfrak{g},\sigma)$ is Abelian it is said to be of Euclidean type.
\end{enumerate}
\end{definition}

\begin{definition}[irreducible OSAKA]
An OSAKA $\left(\widehat{L}(\mathfrak{g}, \sigma), \widehat{\rho}\right)$ is called irreducible iff its derived algebra has no non-trivial derived Kac-Moody subalgebra invariant under $\widehat{\rho}$.
\end{definition}

\noindent Thus we can describe the different classes of irreducible OSAKAs of compact type.

\begin{enumerate}
\item The first class consists of pairs consisting of compact real forms $\widehat{M\mathfrak{g}}$, where $\mathfrak{g}$ is a simple Lie algebra together with an involution of the second kind. Complete classifications are available~\cite{Heintze08}. This are irreducible factors of type $I$. They correspond to Kac-Moody symmetric spaces of type $I$.
\item Let $\mathfrak{g}_{\mathbb{R}}$ be a simple real Lie algebra of the compact type. The second class consists of  pairs of an affine Kac-Moody algebra $\widehat{M(\mathfrak{g}_{\mathbb{R}}\times \mathfrak{g}_{\mathbb{R}})}$ together with an involution of the second kind, interchanging the factors. Those algebras correspond to Kac-Moody symmetric spaces that are compact Kac-Moody groups equipped with their $Ad$\ndash invariant metrics (type $II$). 
\end{enumerate}

\noindent Dualizing OSAKAs of the compact type, we get the OSAKAs of the noncompact type.

\begin{enumerate}
\item Let $\mathfrak{g}_{\mathbb{C}}$ be a complex semisimple Lie algebra, and $\widehat{L}(\mathfrak{g}_{\mathbb{C}},\sigma)$ the associated affine Kac-Moody algebra.
This class consists of real forms of the noncompact type that are described as fixed point sets of involutions of type 2 together with a special involution, called Cartan involution. This is the unique involution on $\mathcal{G}$, such that the decomposition into its $\pm 1$ eigenspaces $\mathcal{K}$ and $\mathcal{P}$ yields: $\mathcal{K}\oplus i \mathcal{P}$ is a real form of compact type of $\widehat{L}(\mathfrak{g}_{\mathbb{C}},\sigma)$. Those orthogonal symmetric Lie algebras correspond to Kac-Moody symmetric spaces of type $III$.

\item Let $\mathfrak{g}_{\mathbb{C}}$ be a complex semisimple Lie algebra. The fourth class consists of $\widehat{L}(\mathfrak{g}_{\mathbb{C}},  \sigma)$ with the involution given by the complex conjugation $\widehat{\rho_0}$ with respect to a compact real form, i.e.\ $\widehat{L}(\mathfrak{g}_{\mathbb{R}}, \sigma)$. 
Those algebras correspond to Kac-Moody symmetric spaces of type $IV$.
\end{enumerate}

The derived algebras of the last class of OSAKAs --- the ones of Euclidean type  --- are Heisenberg algebras~\cite{PressleySegal86}. The maximal subgroup of compact type is trivial.

\subsection{Affine Kac-Moody groups}
\label{section:Affine_Kac-Moody_groups}

There are several different approaches to affine Kac-Moody groups. The usual algebraic approach follows the definition of algebraic groups via a functor. Tits defines a group functor from the category of rings into the category of groups, whose evaluation on the category of fields yields Kac-Moody groups. Various completions of the groups defined this way are possible. Nevertheless, restricting to affine Kac-Moody groups, there is a second much more down to earth approach which consists in the definition of Kac-Moody groups as special extensions of loop groups - for those two and other approaches compare the article \cite{Tits84}. This second approach for affine Kac-Moody groups relies on a curious identification:
Let $G$ denote an affine algebraic group scheme and study the group $G(k[t,t^{-1}])$, where $k$ is a field. Either, defining a torus extension of $G(k[t,t^{-1}])$ we get a Kac-Moody group over the field $k$ or tensoring with the quotient field $k(t)$ of $k[t,t^{-1}]$ we get an algebraic group over $k(t)$. This hints to a close connection between algebraic groups and affine algebraic Kac-Moody groups, with a loop group as the intermediate object. Now ``analytic'' completions can be defined just by completing the ring $k[t,t^{-1}]$ with respect to some norm.

In this section we focus on the loop group approach to Kac-Moody groups.  Our presentation follows the book~\cite{PressleySegal86}.

We use again the regularity-independent notation $L(G_{\mathbb{C}}, \sigma)$ for the complex loop group and $L(G, \sigma)$ for its real form of compact type. To define groups of polynomial maps, we use the fact, that every compact Lie group is isomorphic to a subgroup of some unitary group. Hence we can identify it with a matrix group. Similarly, the complexification can be identified with a subgroup of some general linear group~\cite{PressleySegal86}.

\noindent Kac-Moody groups are constructed in two steps.
\begin{enumerate}
\item The first step consists in the construction of an $S^1$\ndash bundle in the real case (resp.\ a $\mathbb{C}^*$\ndash bundle in the complex case) over $L(G, \sigma)$ that corresponds via the exponential map to the derived algebra. The fiber corresponds to the central term $\mathbb{R}c$ (resp.\ $\mathbb{C}c$) of the Kac-Moody algebra.
\item In the second step we construct a semidirect product with $S^1$ (resp.\ $\mathbb{C}^*$). This corresponds via the exponential map to the $\mathbb{R}d$\ndash\/ (resp.\ $\mathbb{C}d$\ndash) term
\end{enumerate}

\noindent Study first the extension of $L(G, \sigma)$ with the short exact sequence:
$$1 \longrightarrow S^1\longrightarrow X \longrightarrow L(G, \sigma) \longrightarrow 1\,.$$

\noindent There are various groups $X$ that fit into this sequence. We need to find a group $\widetilde{L}(G, \sigma)$ such that its tangential Lie algebra at $e\in \widetilde{L}(G, \sigma)$  is isomorphic to $\widetilde{L}(\mathfrak{g}, \sigma)$.

As described in~\cite{PressleySegal86} this $S^1$\ndash bundle is best represented by triples $(g,p, z)$ where $g$ is an element in the loop group, $p$ a path connecting the identity to $g$ and $z\in  S^1 (\textrm{respective } \mathbb C^*)$ subject to the relation of equivalence: $(g_1, p_1, z_1) \sim (g_2, p_2, z_2)$ iff $g_1= g_2$ and $z_1= C_{\omega}(p_2*p_1^{-1})z_2$. Here $C_{\omega}(p_2*p_1^{-1})=e^{\int_{S(p_2*p_1^{-1})}\omega}$ where $S(p_2*p_1^{-1})$ is a surface bounded by the closed curve $p_2*p_1^{-1}$ and $\omega$ denotes the $2$\ndash form used to define the central extension of $L(\mathfrak{g}, \sigma)$.  The term $z_1= C_{\omega}(p_2*p_1^{-1})z_2$ defines a twist of the bundle. The law of composition is defined by $$(g_1, p_1, z_1)\cdot (g_2, p_2, z_2)=(g_1g_2, p_1*g_1(p_2), z_1 z_2)\,.$$

If $G$ is simply connected and $\omega$ integral (which is the case in our situation), it can be shown that this object is a well defined group~\cite{PressleySegal86}, theorem 4.4.1.
If $G$ is not simply connected, the situation is a little more complicated: Let $G=H/Z$ where $H$ is a simply connected Lie group and $Z=\pi_1(G)$. Let $(LG)_0$ denote the identity component of $LG$. We can describe the  extension using the short exact sequence~\cite{PressleySegal86}, section 4.6. :
$$1\longrightarrow S^1\longrightarrow \widetilde{LH}/Z\longrightarrow (LG)_0\longrightarrow 1$$

\noindent In case of complex loop groups, the $S^1$\ndash bundle is complexified to a $\mathbb{C}^*$\ndash bundle.

The second much easier extension yields now Kac-Moody groups:

\begin{definitions}[Kac-Moody group]~
\begin{enumerate}
\item Let $G$ be a compact real Lie group. The compact real Kac-Moody group $\widehat{L}(G, \sigma)$ is the semidirect product of
$S^1$ with the $S^1$\ndash bundle $\widetilde{L}(G, \sigma)$.

\item Let $G_{\mathbb{C}}$ be a complex simple Lie group. The complex Kac-Moody group $\widehat{L}(G_{\mathbb{C}}, \sigma)$ is the semidirect product
of $\mathbb C^*$ with the $\mathbb{C}^*$\ndash bundle $\widetilde{L}(G_{\mathbb{C}}, \sigma)$. It is the complexification of $\widehat{L}(G, \sigma)$.
\end{enumerate}
\end{definitions}

\noindent The action of the semidirect $S^1$\ndash (resp.\ $\mathbb{C}^*$\ndash) factor is in both cases given by a shift of
the argument: $\mathbb C^* \ni w: MG \rightarrow MG: f(z)\mapsto
f(z \cdot w)$.

\section{Kac-Moody symmetric spaces}
\label{section:Kac-Moody_symmetric_spaces}

The existence of Kac-Moody symmetric spaces was conjectured by Chuu-Lian Terng 1995 in her article~\cite{terng95}, nevertheless she explains that a rigorous definition faces serious difficulties. In his thesis Bogdan Popescu investigates possible constructions and is able to define Kac-Moody symmetric spaces of the compact type; nevertheless his approach fails for spaces of the noncompact type~\cite{Popescu05}. 

While we have freedom in choosing which regularity to use for the construction of many objects in Kac-Moody geometry, the regularity of Kac-Moody symmetric spaces is completely fixed by the algebraic operations required: we described that the semidirect factor acts on the loop group by a shift of the argument. As in the complex case this factor is isomorphic to $\mathbb{C}^*$, we need holomorphic loops on $\mathbb{C}^*$ for the complexification.

The realizations, we use for Kac-Moody symmetric spaces are the non-twisted groups $MG:=\{f:\mathbb{C}^*\longrightarrow G_{\mathbb{C}}|f \textrm{ is holomorphic}\}$ and the twisted groups $MG^{\sigma}:=\{f\in MG |\ \sigma \circ f(z)=f(\omega z)\}$. The real form of compact type is defined by $MG_{\mathbb R}:=\{f\in MG_{\mathbb C}| f(S^1) \subset G_{\mathbb R}\}\,.$

\begin{theorem}
\label{mgcistame}
$MG_{\mathbb C}$ and $MG_{\mathbb{C}}^{\sigma}$ are tame Fréchet manifolds. The same is true for compact real forms and all quotients that appear in the definition of Kac-Moody symmetric spaces.
\end{theorem}

The idea of the proof is to use logarithmic derivatives. The concept of logarithmic derivatives for regular Lie groups is developed in the book~\cite{Kriegl97}, chapters 38 and 40. 
Furthermore it is used by Karl-Hermann Neeb to prove regularity results for locally convex Lie groups~\cite{Neeb06}.
Recall the definitions of a tame Fréchet space~\cite{Hamilton82}:

\begin{definition}[Fréchet space]
A Fréchet vector space is a locally convex topological vector space which is complete, Hausdorff and metrizable.
\end{definition}

\begin{definition}[Grading]
Let $F$ be a Fréchet space. A grading on $F$ is a collection of seminorms $\{\|\hspace{3pt}\|_{n}, n\in \mathbb N_0\}$ that define the topology and satisfy
$$\|f \|_0\leq \|f\|_1 \leq \|f\|_2 \leq\| f \|_3 \leq \dots \,.$$
\end{definition}

\begin{definition}[Tame equivalence of gradings]
Let $F$ be a graded Fréchet space, $r,b \in \mathbb{N}$ and $C(n), n\in \mathbb{N}$ a sequence with values in $\mathbb{R}^+$. The two gradings  $\{\|\hspace{3pt}\|_n\}$ and $\{\widetilde{\|\hspace{3pt}\|}_n\}$ are called $\left(r,b,C(n)\right)$-equivalent iff 
\begin{equation*}
 \|f\|_n \leq C(n) \widetilde{\|f\|}_{n+r} \text{ and }  \widetilde{\|f\|}_n \leq C(n)\|f\|_{n+r} \text{ for all } n\geq b
\,.
\end{equation*}
They are called tame equivalent iff they are $(r,b,C(n))$\ndash equivalent for some $(r,b,C(n))$.
\end{definition}

\begin{example}
Let $B$ be a Banach space with norm $\| \hspace{3pt} \|_B$. Denote by $\Sigma(B)$ the space of all exponentially decreasing sequences $\{f_k\}$, ${k\in \mathbb N_0}$ of elements of $B$.
On this space, we can define various different gradings; among them are the following:
\begin{equation*}
\begin{aligned}
\|f\|_{l_1^n} &:= \sum_{k=0}^{\infty}e^{nk} \|f_k\|_B\\
\|f\|_{l_{\infty}^n}&:= \sup_{k\in \mathbb N_0} e^{nk}\|f_k\|_B
\end{aligned}
\end{equation*}
\end{example}

\begin{definition}[Tame map]
A linear map $\varphi: F\longrightarrow G$ is called $(r,b,C(n))$\ndash tame if it satisfies the inequality
$$\|\varphi(f)\|_n \leq C(n)\|f\|_{n+r}\,.$$
$\varphi$ is called tame iff it is $(r,b,C(n))$\ndash tame for some $(r,b, C(n))$.
\end{definition}

\begin{definition}[Tame isomorphism]
A map $\varphi:F\longrightarrow G$ is called a tame isomorphism iff it is a linear isomorphism and $\varphi$ and $\varphi^{-1}$ are tame maps.
\end{definition}

\begin{definition}[Tame direct summand]
$F$ is a tame direct summand of $G$ iff there exist tame linear maps $\varphi: F\longrightarrow G$ and $\psi: G \longrightarrow F$ such that $\psi \circ \varphi: F \longrightarrow F$ is the identity.
\end{definition}

\begin{definition}[Tame  space]
$F$ is tame iff there is a Banach space $B$ such that $F$ is a tame direct summand of $\Sigma(B)$.
\end{definition}

\begin{theorem}
The space $F: =\textrm{Hol}(\mathbb{C}^*,\mathbb C^n)$ is a tame Fréchet space.
\end{theorem}

\noindent The quite technical proof can be found in \cite{Freyn09}.  As a consequence $M\mathfrak{g}$ is a tame Lie algebra (For a Lie algebra to be tame, we require additionally, that the adjoint action of each element is tame).

\begin{proof} Proof of theorem \ref{mgcistame}~\cite{Freyn09}: We concentrate on the special case $MG$.

\noindent Start with an embedding
\[
\begin{array}{ccr}
\varphi: MG_{\mathbb{C}}&\hookrightarrow& \Omega^{1}(\mathbb{C}^*, \mathfrak{g}_{\mathbb{C}}) \times G_{\mathbb{C}}\\
f &\mapsto& (\delta(f)=f^{-1}df, f(1))
\end{array}
\]

Let $\pi_1$ and $\pi_2$ denote the projections:
\[
\begin{array}{ccc}
\pi_1: \Omega^{1}(\mathbb{C}^*, \mathfrak{g}_{\mathbb{C}}) \times G_{\mathbb{C}} &\mapsto& \Omega^{1}(\mathbb{C}^*, \mathfrak{g}_{\mathbb{C}})\\
\pi_2: \Omega^{1}(\mathbb{C}^*, \mathfrak{g}_{\mathbb{C}}) \times G_{\mathbb{C}} &\mapsto& G_{\mathbb{C}}
\end{array}
\]

Charts for $MG_{\mathbb{C}}$ will be products of charts for $\pi_1 \circ \varphi(MG_{\mathbb{C}})$ and $\pi_2 \circ \varphi(MG_{\mathbb{C}})$.
\begin{itemize}
\item[-] $\pi_2 \circ \varphi$ is surjective; hence to describe the second factor, we can choose charts for $G$. Via the exponential mapping, we use charts in $\mathfrak{g}_{\mathbb{C}}$. To describe the norms, we use for $\|\hspace{3pt}\|_n$ on this factor the Euclidean metric.  

\item[-] The first factor is more difficult to deal with as $\pi_1 \circ \varphi$ is not surjective. While every form $\alpha \in \Omega(\mathbb{C}^*, \mathfrak{g}_{\mathbb{C}})$ is locally integrable, the monodromy may prevent global integrability. 
A form $\alpha \in \Omega^1(\mathbb{C}^*, \mathfrak{g}_{\mathbb{C}})$ is in the image of $\pi_1 \circ \varphi$ iff its monodromy vanishes, that is iff
$$e^{\int_{\mathbb{S}^1}\alpha}=e\in G_{\mathbb{C}}\,.$$

This is equivalent to the condition $\int_{\mathbb{S}^1}\alpha = a_{-1}(\alpha) \subset \frac{1}{2\pi i} \exp^{-1}(e)$ where $a_{-1}(\alpha) $ denotes the $(-1)$\ndash  coefficient of the Laurent development  of $\alpha= f(z)dz$. So we can describe $\Im(\pi_1\circ \varphi)$ as the inverse image of the monodromy map of $e\in G_{\mathbb{C}}$.

Thus we have to show that this inverse image is a tame Fréchet manifold. To this end, we use composition with a chart $\psi: U\longrightarrow V$ for $e\in U\subset G$ with values in $G_{\mathbb{C}}$. This gives us a tame map $\Omega(\mathbb{C}^*, \mathfrak{g}_{\mathbb{C}})\longrightarrow \mathfrak{g}_{\mathbb{C}}$. The proof is completed by showing that this map is regular and the proof that inverse images of regular maps are tame Fréchet submanifolds~\cite{Freyn09}.  This proves that $\pi_1 \circ \varphi$ is a tame Fréchet submanifold.
\end{itemize}
\noindent Thus $MG_{\mathbb{C}}$ is a product of a tame Fréchet manifold with a Lie group and hence a tame Fréchet manifold.
This completes the proof of theorem~\ref{mgcistame}.
\end{proof}

A similar result for Kac-Moody groups, is a direct consequence of a result of Bogdan Popescu, stating that fiber bundles whose fiber is a Banach space over tame Fréchet manifolds are tame~\cite{Popescu05}.

Kac-Moody symmetric spaces are tame Fréchet manifolds. Those manifolds have tangential spaces, that can be equipped with weak metrics. We suppose furthermore the existence of a Levi-Civita connection. Then one can prove, that the usual algebraic identities  --- i.e.\ the Bianchi identities --- hold~\cite{Hamilton82}. 

Other concepts well known from the finite dimensional case can be generalized too: 

\noindent A nice example is Kulkarni's theorem about the sectional curvature on Lorentz manifolds:
Define as usual the sectional curvature by $K_f(g,h)=\frac{\langle R(f)\{g,h,g\},h \rangle}{|g\wedge h|^2}$ if  $|g\wedge h|^2 \neq 0 $. Then

\begin{theorem}[generalized Kulkarni-type]
\label{Kulkarni}
Let $M$ be a tame Fréchet Lorentz manifold with Levi-Civita connection. Then the
following conditions are equivalent:
\begin{itemize}
  \item[-] \hspace{3pt}  $K_f(g,h)$ is constant.
  \item[-] \hspace{3pt} $a\leq K_f(g,h)$ or $K_f(g,h) \leq b$ for some $a, b\in \mathbb R$.
  \item[-] \hspace{3pt} $a \leq K_f(g,h) \leq b$ on all definite planes for some
    $a\leq b \in \mathbb R$.
  \item[-] \hspace{3pt} $a \leq K_f(g,h) \leq b$ on all indefinite planes for some
    $a\leq b \in \mathbb R$.
\end{itemize}
\end{theorem}

For the finite dimensional proof of this theorem and more generally finite dimensional Lorentz geometry~\cite{Oneill83}.

Let us now turn to Kac-Moody symmetric spaces themselves:

\begin{definition}[Kac-Moody symmetric space]
An (affine) Kac-Moody symmetric space $M$ is a tame Fréchet Lorentz symmetric space such that its isometry group $I(M)$ contains a transitive subgroup isomorphic to an affine geometric Kac-Moody group $H$. 
\end{definition}

We can distinguish Kac-Moody symmetric spaces of the Euclidean, the compact and the noncompact type, corresponding to the respective types of Riemannian symmetric spaces:

We have the following results:

\begin{theorem}[affine Kac-Moody symmetric spaces of the ``compact'' type]

Both the  Kac-Moody group $\widehat{MG}^{\sigma}_{\mathbb R}$
equipped with its $Ad$\ndash invariant metric, and the quotient space
$X= \widehat{MG}^{\sigma}_{\mathbb R}/\textrm{Fix}(\rho_*)$ equipped with their
$Ad$\ndash invariant metric are tame Fréchet
symmetric spaces of the ``compact'' type with respect to their natural $Ad$\ndash invariant
metric. Their curvatures satisfy 
\[
\langle R(X,Y)X, Y\rangle \geq 0\,.
\]
\end{theorem}

\begin{theorem}[affine Kac-Moody symmetric spaces of the ``noncompact'' type]

Both quotient spaces $X=\widehat{MG}^{\sigma}_{\mathbb{C}}/\widehat{MG}^{\sigma}_{\mathbb R}$ and  
$X=H/\textrm{Fix}(\rho_*)$, where $H$ is a noncompact real form of $\widehat{MG}^{\sigma}_{\mathbb{C}}$ equipped with their $Ad$\ndash invariant metric,
are tame Fréchet symmetric spaces of the ``noncompact'' type. Their curvatures satisfy

\[\langle R(X,Y)X, Y\rangle \leq 0\,.
\]

\noindent Furthermore Kac-Moody symmetric spaces of the noncompact type are
diffeomorphic to a vector space.
\end{theorem}

\noindent Define the notion of duality as for finite dimensional Riemannian symmetric spaces.

\begin{theorem}[Duality]
Affine Kac-Moody symmetric spaces of the compact type are dual to Kac-Moody symmetric spaces of the noncompact type and vice versa.
\end{theorem}

A complete classification of Kac-Moody symmetric spaces following the lines of the classification of finite dimensional Riemannian symmetric spaces follows from the correspondence between simply connected Kac-Moody symmetric spaces and OSAKAs~\cite{Freyn09}. 

Kac-Moody symmetric spaces have several conjugacy classes of flats. For our purposes the most important class are those of finite type. A flat is called of finite type iff it is finite dimensional. A flat is called of exponential type iff it lies in the image of the exponential map and it is called maximal iff it is not contained in another flat. Adapting a result of Bogdan Popescu~\cite{Popescu05} to our setting, and generalizing it to symmetric spaces of the noncompact type, we find:

\begin{theorem}
All maximal flats of finite exponential type are conjugate.
\end{theorem}

The isotropy representations of Kac-Moody symmetric spaces correspond exactly to the $s$\ndash representations for involutions on affine Kac-Moody algebras that are studied in \cite{Gross00} and induce hence polar actions on Fréchet\ndash\ resp.\ Hilbert spaces. This is the subject of the next section.

\section{Polar actions}
\label{section:Polar_actions_on_Hilbert_and_Frechet_spaces}

 Let $P(G, H):= \left\{g\in \hat{G}= H^1([0,1], G)\mid (g(0), g(1))\in H\subset G \times G\right\}$ denote the space of all $H^1$-Sobolev path in $G$ whose endpoints are in $H$ and let $V=H^0([0,1], \mathfrak{g})$ denote the space of all $H^0$-Sobolev path in $V$.

\noindent We quote the following theorem of \cite{terng95}:

\begin{theorem}
 Suppose the action of $H$ on $G$ is hyperpolar. Let $A$ be a torus section through $e$ and let $\mathfrak{a}$ denote its Lie algebra. Then
\begin{enumerate}
 \item the $P(G,H)$-action on $V$ is polar and the space $\widehat{\mathfrak{a}}=\{\widehat{a}|a\in \mathfrak{a}\}$ is a section, where $\widehat{a}:[0,1]\longrightarrow \mathfrak{g}$ denotes the constant map with value $a$. 
\item Let $N(\mathfrak{a})$ be the normalizer of $\mathfrak{a}$ in $P(G, H)$, $Z(\mathfrak{a})$ the centralizer. The quotient $N(\mathfrak{a})/Z(\mathfrak{a})$ of the normalizer of $\mathfrak{a}$ is an affine Weyl group.
\end{enumerate}
\end{theorem}

Similar to finite dimensional polar representations, that are induced by the isotropy representations of symmetric spaces, the classical examples of polar actions on Hilbert space come from isotropy representations of Kac-Moody symmetric spaces:

\begin{theorem}
 Let $\widehat{L}(G, \sigma)$ be a Kac-Moody group of $H^1$\ndash loops and let $\rho:\widehat{L}(G, \sigma)\longrightarrow \widehat{L}(G, \sigma)$ be an involution of the second kind i.e.\ an involution such that on the Lie algebra level $d\rho(c)=-c$ and $d\rho(d)=-d$. Let $L(\mathfrak{g}, \sigma)=\mathcal{K}\oplus \mathcal{P}$ be the decomposition of $L(\mathfrak{g}, \sigma)$ into the $\pm 1$\ndash eigenspaces of $d\rho$. Then the restriction of the Adjoint action of $\textrm{Fix}(\rho)\subset \widehat{L}(G, \sigma)$ to the subspace 
$\{x \in \mathcal{P}| r_d=1$ and $|x|=-1\}$ where $r_d$ denotes the coefficient of $d$, is polar.
\end{theorem}

\noindent For the definition of involutions of the second kind and a classification of involutions see \cite{Heintze08} and \cite{Heintze09}.
\begin{proof}
~\cite{Gross00}
\end{proof}

We describe the four classes of $P(G,H)$\ndash actions more explicitly~\cite{HPTT}:
\begin{enumerate}
 \item The diagonal subgroup $H:=\{(x,x)\in G\times G\}$. Elements of $P(G,H)$ are closed loops. The polar actions defined by this group are induced by the isotropy representations of non-twisted Kac-Moody symmetric spaces of type II resp.\ type IV.
\item The twisted diagonal subgroup $H:=\{(x, \sigma(x))\in G\times G\}$, where $\sigma$ is a diagram automorphism. The associated polar actions are induced by the isotropy representations of twisted Kac-Moody symmetric spaces of types II resp.\ IV.
\item Let $H:=K\times K\subset G\times G$. The polar actions defined by these groups are induced by the isotropy representations of non-twisted Kac-Moody symmetric spaces of type I resp.\ type III.
\item Let $H=K_{\sigma_1}\times K_{\sigma_2}$ where $K_{\sigma_i}$ are fixed point groups of suitable involutions. The polar actions defined by these groups are induced by the isotropy representations of twisted Kac-Moody symmetric spaces of type I resp.\ type III.
\end{enumerate}

\section{Isoparametric submanifolds}
\label{section:Isoparametric_submanifolds}

\begin{definition}[$PF$-Submanifold]
 An immersed finite codimensional submanifold $M\subset V$ is proper Fredholm (PF) if the restriction of the end point map to a disk normal bundle of $M$ on any radius $r$ is proper Fredholm.
\end{definition}

\begin{definition}
 An immersed PF submanifold $f:M\longrightarrow V$ of a Hilbert space $V$ is called isoparametric if
\begin{enumerate}
 \item $codim(M)$ is finite.
 \item $\nu(M)$ is globally flat.
 \item for any parallel normal field $v$ on $M$, the shape operators $A_{v(x)}$ and $A_{v(y)}$ are orthogonally equivalent for all $x,y$ in $M$ 
\end{enumerate}
\end{definition}

The two main references are \cite{Terng89} and \cite{PalaisTerng88}. One can associate to isoparametric submanifolds in Hilbert spaces affine Weyl groups, that describe --- as in the finite dimensional case --- the positions of the curvature spheres.
\noindent Similar to the finite dimensional case, most examples for isoparametric submanifolds in Hilbert space arise from polar actions:

\begin{theorem}
 A principal orbit of a polar action on a Hilbert space is isoparametric.
\end{theorem}

\noindent A partial converse also exists:

\begin{theorem}[Heintze-Liu]
 A (complete, connected, irreducible, full) isoparametric submanifold of a Hilbert space $V$ with codimension $\not= 1$ ($\not=2$ if dim $V < \infty$) is a principal orbit of a polar action.
\end{theorem}

\noindent For the proof see  \cite{HeintzeLiu}.
In principle a proof along the lines of Thorbergsson should be possible also.

\section{Universal geometric twin buildings}
\label{section:Universal_geometric_twin_buildings}

For Kac-Moody symmetric spaces universal twin buildings take over the role played by spherical buildings for finite dimensional Riemannian symmetric spaces.

For algebraic loop groups (resp.\ affine Kac-Moody groups) there appear two constructions
of buildings in the literature:

\begin{itemize}
\item[-] If one replaces semisimple Lie groups (or more generally
reductive linear algebraic groups) by algebraic Kac-Moody groups, it
turns out that algebraic twin buildings take over the role played by spherical buildings for Lie groups.

\noindent As in the Lie group situation we want an equivalence between Borel subgroups of the Kac-Moody group and chambers of the building. Due to the fact that Kac-Moody groups have two conjugacy classes of Borel subgroups the associated ``building'' breaks up into two pieces: Such a a twin building consists of a pair  $\mathfrak{B}^+\cup \mathfrak{B}^-$ of buildings
that are ``twinned'': The twinning can be described in different ways: From the
point of view of apartments the twinning is most easily defined by the
introduction of a system of twin apartments, that is subcomplexes
$\mathcal{A}^+\cup \mathcal{A}^-\subset \mathfrak{B}^+\cup \mathfrak{B}^-$, consisting of
two apartments $\mathcal{A}^+$ and $\mathcal{A}^-$, such that $\mathcal{A}^+$ is contained in $\mathfrak{B}^+$ and $\mathcal{A}^-$ is contained in $\mathfrak{B}^-$.
Imposing some axioms similar to those used for spherical buildings, many 
features known from apartments in spherical buildings generalize to the
system of twin apartments~\cite{AbramenkoBrown08}. 

\item[-] For groups of algebraic loops there is a theory of affine buildings (but not
for twin buildings) developed from the point of view of loop groups~\cite{Mitchell88}. 
\end{itemize}

To associate twin buildings to Kac-Moody symmetric spaces the main
problem is again to unify algebraic and analytic aspects of the
theory: twin buildings associated to affine Kac-Moody groups consist
of pairs of Euclidean buildings. These are purely algebraic
constructions and thus they correspond only to the subgroup of algebraic
loops.  Written in an algebraic notion these groups are of the form $G(\mathbb{C}[z, z^{-1}])$, that is: groups of polynomial loops in $z$ and $z^{-1}$. 
Their affine Weyl groups act transitively on the chambers of
the apartments whereas the groups $G(\mathbb{C}[z, z^{-1}])$ act transitively
on the apartments.

\noindent Unfortunately a straightforward process of completion ---
even in only one direction $z$ or $z^{-1}$ --- destroys the twinned structure. The
classically used remedy can be found in Shrawan Kumar's book~\cite{Kumar02}:
A group that is completed only in one direction --- (let's say in the direction of
$z$: hence the Laurent polynomials in $z$ and $z^{-1}$ are replaced by holomorphic
functions with finite principal part) --- acts on the part of the twin
building corresponding to this direction (i.e.\ $\mathfrak{B}^+$). For our
purposes however it is not enough to restrict the theory to one half of the twin
building: On the one hand we need a completion that is symmetric in
$z$ and $z^{-1}$ in order to be able to define the involutions of the second kind
which are needed for Kac-Moody symmetric spaces. On the other hand
our long-term objective is a proof of Mostow rigidity and there are 
good rigidity results only for twin buildings themselves, but not for their affine
parts separately.

Our solution in \cite{Freyn09} is to define on the level of the building a
``completion'' of the two parts of the twin building that corresponds
to the completed groups. For different completions of the loop group (i.e.\ 
    $H^1$, $C^{\infty}$, analytic or holomorphic loops), the associated completions
of the building are different. 
\noindent The process of completing algebraic twin buildings leads to a simplicial complex
$\mathfrak{B}= \mathfrak{B}^+ \cup \mathfrak{B^-}$, whose (uncountable
    many) positive and negative connected components are affine buildings;
we call these objects ``universal geometric buildings''.

\subsection{Universal BN-pairs}

\begin{definition}[Universal geometric BN\ndash pair for $\widehat{L}(G,\sigma)$]

Let $\widehat{L}(G,\sigma)$ be an affine Kac-Moody group. 
Data $(\widehat{B}^+, \widehat{B}^-, N, W, S)$  is a twin BN\ndash pair for $\widehat{L}(G,\sigma)$ iff there are subgroups $\widehat{L}(G, \sigma)^+$ and $\widehat{L}(G, \sigma)^-$ of $\widehat{L}(G,\sigma)$ such that $\widehat{L}(G,\sigma)=\langle \widehat{L}(G, \sigma)^-, \widehat{L}(G, \sigma)^+ \rangle$ subject to the following axioms:
\begin{enumerate}
	\item $(\widehat{B}^+, N, W, S)$ is a BN\ndash pair for $\widehat{L}(G,\sigma)^+$ (called $B^+N$),	
	\item $(\widehat{B}^-, N, W, S)$ is a BN\ndash pair for $\widehat{L}(G,\sigma)^-$ (called $B^-N$),
	\item $(\widehat{B}^+\cap \widehat{L}(G, \sigma)^- , \widehat{B}^-\cap \widehat{L}(G, \sigma)^+, N, W, S)$ is a twin BN-pair for     $\widehat{L}(G, \sigma)^+\cap \widehat{L}(G, \sigma)^-$.
\end{enumerate}
\end{definition}

The subgroups $\widehat{L}(G, \sigma)^+$ and $\widehat{L}(G, \sigma)^-$of $\widehat{L}(G,\sigma)$ depend on the choice of $\widehat{B}^+$ and $\widehat{B}^-$. A choice of a different subgroup $\widehat{B}^{+'}$ (resp.\  $\widehat{B}^{-'}$) gives the same  subgroup $\widehat{L}(G, \sigma)^+$ (resp.\ $\widehat{L}(G, \sigma)^-$) of $\widehat{L}(G,\sigma)$ if $\widehat{B}^{+'}\subset \widehat{L}(G, \sigma)^+$ (resp. if $\widehat{B}^{-'}\subset \widehat{L}(G, \sigma)^-)$). For all positive (resp.\ negative) Borel subgroups the positive (resp.\ negative) subgroups $\widehat{L}(G, \sigma)^+$ resp $\widehat{L}(G, \sigma)^-$ are conjugate. Hence without loss of generality we can think of $\widehat{B}^{\pm}$ to be the standard positive (resp negative) affine Borel subgroup. The groups $\widehat{L}(G, \sigma)^{\pm}$ -- called the standard positive (resp.\ negative) subgroups are then characterized by the condition that $0$ (resp.\ $\infty$) is a pole for all elements.

\begin{remark}
For an algebraic Kac-Moody group a generalized BN\ndash pair coincides with a BN\ndash pair. Hence we get $\widehat{L}(G,\sigma)^+ =\widehat{L}(G,\sigma)^-=\widehat{L}(G,\sigma)$.
\end{remark}

\noindent We use the equivalent definition for the loop groups $L(G, \sigma)$.

\begin{lemma}
\label{bruhatforloopgroups}
\begin{enumerate}
  \item The groups $L(G,\sigma)^+$ (resp.\ $L(G,\sigma)^-$) have a positive (resp.\ negative) Bruhat decomposition and a Bruhat twin decompositions.
  \item The group  $L(G,\sigma)$ has a Bruhat twin decomposition but no Bruhat decomposition.
\end{enumerate}
\end{lemma}

\begin{proof}
The Bruhat decompositions in the first part follow by definition, the Bruhat twin decomposition by restriction and the second part.
The second part is a restatement of the decomposition results in chapter $8$ of~\cite{PressleySegal86}. 
\end{proof}

\noindent Compare also the similar decomposition results stated in \cite{Tits84}.

\begin{theorem}[Bruhat decomposition]
Let $\widehat{L}(G, \sigma)$ be an affine Kac-Moody group with affine Weyl group $W_{\textrm{aff}}$.
Let furthermore $\widehat{B}^\pm$ denote a positive (resp.\ negative) Borel group. There are decompositions

$$\widehat{L}(G, \sigma)^+=\coprod_{w\in W_{\textrm{aff}}}\widehat{B}^+w\widehat{B}^+\hspace{10pt}
\textrm{and}\hspace{10pt}
\widehat{L}(G, \sigma)^-=\coprod_{w\in W_{\textrm{aff}}}\widehat{B}^-w\widehat{B}^-\,.
$$
\end{theorem}

\begin{proof}
This is a consequence of lemma \ref{bruhatforloopgroups}.
\end{proof}

\begin{theorem}[Bruhat twin decomposition]
Let $\widehat{L}(G, \sigma)$ be an affine algebraic Kac-Moody group with affine Weyl group $W_{\textrm{aff}}$.
Let furthermore $\widehat{B}^\pm$ denote a positive and its opposite negative Borel group. There are two decompositions
 $$\widehat{L}(G,\sigma)=\coprod_{w\in W_{\textrm{aff}}}\widehat{B}^{\pm}w\widehat{B}^{\mp}\,.$$ 
\end{theorem} 

\begin{remark}
Note that the Bruhat twin decomposition is defined on the whole group $\widehat{L}(G,\sigma)$. This translates into the fact that any two chambers in $\mathfrak{B}^+$ resp.\ $\mathfrak{B}^-$ have a well-defined Weyl codistance. In contrast Bruhat decomposition are only defined for subgroups. This translates into the fact that there are positive (resp.\ negative) chambers without a well-defined Weyl distance.
\end{remark}

\begin{example}
\label{Kumarsexample}
Kumar studies Kac-Moody groups and algebras that are completed ``in one direction''. In the setting of affine Kac-Moody groups of holomorphic loops this means: holomorphic functions with finite principal part. There is an associated twin BN-pair; the positive Borel subgroups are completed affine Borel subgroups while the negative ones are the algebraic affine Borel subgroups. Thus for a universal geometric twin BN-pair we have to use: $\widehat{L}(G,\sigma)^+=\widehat{L}(G,\sigma)$ and $\widehat{L}(G,\sigma)^-=\widehat{L_{alg}G}^{\sigma}$~\cite{Kumar02}.
\end{example}

\begin{lemma}
The intersection $\widehat{L}(G,\sigma)^0$ of  $\widehat{L}(G,\sigma)^+$ with  $\widehat{L}(G,\sigma)^-$ is isomorphic to the group of algebraic loops
$$\widehat{L}(G,\sigma)^0\simeq \widehat{L_{alg}G}^{\sigma}$$
\end{lemma}

\begin{proof}
$\widehat{L_{alg}G}^{\sigma}$ is the maximal subgroup of $\widehat{L}(G,\sigma)$ having both Bruhat decompositions.
\end{proof}

\subsection{Universal geometric twin buildings}

\noindent We now define a universal geometric twin building using the $W$\ndash metric approach:

\begin{definition}[Universal geometric twin building]
Let $\widehat{L}(G,\sigma)$ be an affine Kac-Moody group with a universal geometric BN\ndash pair.
Define $\mathcal{C}^+:=\widehat{L}(G,\sigma)/B^+$ and $\mathcal{C}^-:=\widehat{L}(G,\sigma)/B^-$. 
\begin{enumerate}
\item The distance $\delta^{\epsilon}:\mathcal{C}^{\epsilon}\times \mathcal{C}^{\epsilon} \longrightarrow W$, $\epsilon \in \{+,-\}$ is defined  via the Bruhat decompositions:
$\delta^{\epsilon}(gB^{\epsilon} , fB^{\epsilon})=w$ iff $g^{-1}f$ is in the $w$\ndash class of the Bruhat decomposition of $\widehat{L}(G,\sigma)^{\epsilon}$. Otherwise it is $\infty$.
\item The codistance  $\delta^*:\mathcal{C}^+\times \mathcal{C}^- \cup \mathcal{C}^-\times \mathcal{C}^+ \longrightarrow W$ is defined  by $\delta^*(gB^- , fB^+)=w$ (resp.\ $\delta^*(gB^+ , fB^-)=w$) iff $g^{-1}f$ is in the $w$\ndash class of the corresponding Bruhat twin decomposition of $\widehat{L}(G,\sigma)$.
\end{enumerate}
\end{definition}

The elements of $\mathcal{C}^{\pm}$ are called the positive (resp.\ negative) chambers of the universal geometric twin building. The building is denoted by $\mathfrak{B}=\mathfrak{B^+}\cup\mathfrak{B}^-$. One can define a simplicial complex realization in the usual way. We define connected components in $\mathfrak{B}^{\pm}$ in the following way: Two elements $\{c_1, c_2\}\in \mathfrak{B}^{\pm}$ are in the same connected component iff $\delta^{\pm}(c_1, c_2)\in W_{\textrm{aff}}$. This is an equivalence relation. Denote the set of connected components by $\pi_0(\mathfrak{B})$ resp.\ $\pi_0(\mathfrak{B}^{\pm})$.

\begin{remark}
Let $L(G, \sigma)$ be an algebraic affine Kac-Moody group. Then the universal geometric twin building coincides with the algebraic twin building.
\end{remark}

\begin{lemmas}[Properties of a universal geometric twin building]~
\label{propertiesofgeometricuniversaltwinbuilding}
\begin{enumerate}
\item The connected components of $\mathfrak{B}^{\epsilon}$ are buildings of type $(W,S)$,
\item Each pair consisting of one connected component in $\mathfrak{B}^+$ and one in $\mathfrak{B}^-$ is an algebraic twin building of type $(W,S)$,
\item The connected components of $\mathfrak{B}^{\epsilon}$ are indexed by elements in $\widehat{L}(G, \sigma)/\widehat{L}(G,\sigma)^{\epsilon}$,
\item The action of $\widehat{L}(G, \sigma)$  on $\mathfrak{B}$ by left multiplication is isometric,
\item The Borel subgroups are exactly the stabilizers of the chambers, parabolic subgroups are the stabilizers of simplices,
\item $\widehat{L}(G, \sigma)^{\epsilon}$  acts on the identity component $\Delta^{\epsilon}_0$ by isometries,
\item Let $\Delta_0^+\cup \Delta_0^-$ be a twin building in $\mathfrak{B}$. There is a subgroup isomorphic to the group of algebraic loops, that acts transitively on $\Delta_0^+\cup \Delta_0^-$.
\end{enumerate}
\end{lemmas}

\begin{proof}~
We sketch the proofs of some parts. For additional details~\cite{Freyn10d}
\begin{enumerate}
\item The relation defined by finite codistance is clearly symmetric and self-reflexive, transitivity follows by calculation. One has to check the axioms for a building.
\item Each pair consisting of a connected component in $\mathfrak{B}^+$ and one in $\mathfrak{B}^-$ fulfills the axioms of an algebraic twin building. As the Bruhat decomposition is defined on $\widehat{L}(G, \sigma)$, the codistance is defined between arbitrary chambers in $\mathfrak{B}^{\epsilon}$ resp.\ $\mathfrak{B}^{-\epsilon}$.
\item $\widehat{L}(G,\sigma)$ has a decomposition into subsets of the form $\widehat{L}(G,\sigma)^{\epsilon}$. Those subsets are indexed with elements in $\widehat{L}(G,\sigma)/\widehat{L}(G,\sigma)^{\epsilon}$. The class corresponding to the neutral element is $\widehat{L}(G,\sigma)^{\epsilon}\subset \widehat{L}(G,\sigma)$. Thus it corresponds to a connected component and a building of type $(W,S)$.  The result follows via translation by elements in $\widehat{L}(G,\sigma)/\widehat{L}(G,\sigma)^{\epsilon}$: a connected component of $\mathfrak{B}^{\epsilon}$ containing $fB^{\epsilon}$ consists of all elements $ f\widehat{L}(G,\sigma)^{\epsilon} B^{\epsilon}$ as $\delta(fhB^{\epsilon}, fh'B^{\epsilon})=w((fh)^{-1}fh)= w(h^{-1}f^{-1}fh')=w(h^{-1}h')\in W$ as $h, h'\in \widehat{L}(G, \sigma)^{\epsilon}$. 
\item $G$ acts isometrically on a twin building if the action on both parts preserves the distances and the codistance~\cite{AbramenkoBrown08}, 6.3.1; the result follows by direct calculation. 
\item The chamber corresponding to $fB^{\epsilon}$ is stabilized by the Borel subgroup $B_f^{\epsilon}:=fB^{\epsilon}f^{-1}$. The converse follows as each Borel subgroup is conjugate to a standard one. Analogous for the parabolic subgroups. 
\end{enumerate}

\end{proof}

\begin{theorem} [Embedding of twin buildings]

Denote by $H_{l, r}$ the intersection of the sphere of radius $l$  of a real affine Kac-Moody algebra $\widehat{M\mathfrak{g}}$ with the horospheres $r_d=\pm r$, where $r_d$ is the coefficient of $d$ in the Kac-Moody algebra.
  There is a $2$\ndash parameter family $(l,r)\in \mathbb{R}^+\times \mathbb{R}^+$of $\widehat{MG}$\ndash equivariant immersion of $\mathfrak{B^+\cup \mathfrak{B^-}}$ into  $\widehat{M\mathfrak{g}}$. It is defined by the identification of $H_{l, r}$ with $\mathfrak{B}$. The two complexes
 $\mathfrak{B}^+$ and $\mathfrak{B}^-$ are immersed into the two sheets of $H_{l, r}$.
\end{theorem}

As in the finite dimensional situation, one can extend this result to the isotropy representations of Kac-Moody symmetric spaces. This gives embeddings of the universal geometric twin building into the tangential spaces. A consequence is the following

\begin{cor}
Points of the  isoparametric submanifolds, corresponding to Kac-Moody symmetric spaces,  are in bijection to chambers in one half of the universal twin building 
\end{cor}

This shows in particular that, from a geometric point of view, universal geometric twin buildings
are the correct generalization of spherical buildings for Kac-Moody
groups.

\subsection{Linear representations for universal geometric buildings}

The matrix representations for the classical Lie groups give rise to linear representations for the associated buildings: For example the building of type $A_{n-1}$ over $\mathbb{C}$ corresponds to the flag complex of subspaces in the vector space $V^n:=\mathbb{C}^n$. Buildings for the other classical groups correspond to complexes of \glqq special\grqq\ subspaces~\cite{Garrett97}.

In this section we describe the infinite dimensional flag representation for buildings of type $\widetilde{A}_{n-1}$. As linear representations for the loop groups, we use the operator representations studied in~\cite{PressleySegal86}, chapter 6. Subspaces are elements in suitable Grassmannians \cite{PressleySegal86} (chapter 7). Originally those Grassmannians were introduced by Mikio Sato in the context of integrable systems~\cite{Sato83}.

Let $H^{n}= L^2(S^1, \mathbb{C}^n)$ denote the separable Hilbert space of square summable functions on $S^1$ with values in $\mathbb{C}^n$. Let $H= H^{++} \oplus H^0 \oplus H^{--}$ be a polarization (for example the one induced by the action of $-i\frac{d}{d\theta}$). Set $H^+ = H^{++} \oplus H^0$ and $H^-= H^0 \oplus H^{--}$.

\noindent Following~\cite{PressleySegal86} definition 7.1, the positive Grassmannian is defined as follows:

\begin{definition}[Grassmannian]
\label{positivegrassmanian}
The positive Grassmannian $Gr^{+}(H)$ is the set of all closed subspaces $W$ of $H$ such that
\begin{enumerate}
	\item the orthogonal projection $pr_{+}: W \longrightarrow H^+$ is a Fredholm operator,
	\item the orthogonal projection $pr_{--}: W \longrightarrow H^{--}$ is a Hilbert-Schmidt operator.
\end{enumerate}
\end{definition}

We define the virtual dimension of $W$ to be $\nu (W) =\textrm{dim}(\textrm{ker } pr_+) -\textrm{dim}(\textrm{coker } pr_{+})$.

There are various subgrassmannians corresponding to sundry regularity conditions.
The most important examples are:

\begin{definition}[positive algebraic Grassmannian]
\label{algebraicgrassmanian}    
The positive algebraic Grassmannian $Gr_0^+(H)$ consists of subspaces $W\subset Gr^+(H)$ such that $z^k H^+ \subset W \subset z^{-k} H^+$. 
\end{definition}

Using the explicit description $ H=L^2(S^1, \mathbb{C}^n)$, $Gr_0^+(H)$ consists exactly of the elements $W\in Gr(H)$ such that the images of $pr_{--}:W \longrightarrow H^{--}$ and $pr_+: W^{\perp} \longrightarrow H^+$ are polynomials~\cite{PressleySegal86} and~\cite{Gohberg03}.

\noindent There are various other Grassmannians, i.e.\ the rational Grassmannian $Gr_1(H)$, $Gr_{\omega}(H)$ and the smooth Grassmannian $Gr_{\infty}(H)$~\cite{PressleySegal86}, for the definition of the tame Fréchet Grassmannian $G_{t}$ and the $H^1$\ndash Grassmannian~\cite{Freyn09}, corresponding to different regularity conditions.

\begin{definition}[reduced Grassmannian]
\label{reducedgrassmanian}
The reduced positive Grassmannian $Gr^{n,+}(H^n)$ consists of subspaces $W\subset Gr^+(H^n)$ such that $G(W) \subset W$ 
(or explicitly $zW\subset W$).
\end{definition}

The definition of the other types of reduced Grassmannians, especially reduced algebraic Grassmannians, reduced $H^1$- and reduced tame Grassmannians is self explaining. 

The following theorem~\cite{PressleySegal86}, theorem 8.3.2 --- shows this to be the correct notion to work well with the action of loop groups.

\begin{theorem}
The group $L_{\frac{1}{2}}U_n$ of $\frac{1}{2}$\ndash Sobolev loops acts transitively on $Gr^{n,+}$ and the isotropy group of $H^+$ is the group $U_n$ of constant loops. 
\end{theorem}

This theorem yields the equivalences $\Omega_{\frac{1}{2}}U_n = Gr^{n}(H)$ and $\Omega_{alg}U_n= Gr_0^{n}$. Similar statements hold for $Gr_1^{n,+}(H)$, $Gr_{\omega}^{n,+}(H)$, $Gr_{\infty}^{n,+}(H)$  and $Gr_{t}^{n,+}(H)$; 
for $Gr_1^{n,+}(H)$, $Gr_{\omega}^{n,+}(H)$, $Gr_{\infty}^{n,+}(H)$ a proof can be found in \cite{PressleySegal86}; this proof adapts to the case of $Gr_{t}^{n,+}(H)$ straight forwardly. 

\noindent The next step is the definition of the flag varieties:

As subspaces in a flag satisfy $z W_k = W_{k+n}\subset W_k$, all subspaces of periodic flags are taken from $Gr^{n,+}$.

Let $\{e_1,\dots, e_n\}$ be a basis of $V^n\simeq \mathbb{C}^n$ and $V_i:=\textrm{span}\langle e_{i+1}, \dots, e_n\rangle$. 
We define the positive normal flag to be the flag $\{W_{k'}\}_{k'\in \mathbb{Z}}$ such that for $k'=kn+l, k\in \mathbb{Z}, l\in \{0, \dots, {n-1}\}$ we have $W_{k'}=W_{kn+l}:=z^k W_{l}$   and $W_l$ consists of all functions whose negative Fourier coefficients vanish and whose $0$\ndash coefficient is in $V_i$.

To define the manifolds of partial periodic flags the virtual dimension is used. 
For this purpose let $\mathcal{K}\subset \mathbb{Z}$ be a subset such that with $k\in \mathcal{K}$ also $k+nl\in \mathcal{K}, \forall l\in \mathbb{Z}$. 

\noindent Furthermore set $m_{\mathcal{K}}:= \#\{\mathcal{K}\cap \{1,\dots, n\}\}$. Denote those $m_{\mathcal{K}}$\ndash numbers $k_1, \dots, k_{m_{\mathcal{K}}}$.

\begin{samepage}
\begin{definition}[positive periodic flag manifold]
The positive periodic flag manifold $Fl_{\mathcal{K}}^{n,+}$ consists of all flags $\{W_k\}, k \in \mathbb{Z}$  in $H^{n}$ such that
\begin{enumerate}
	\item $W_k \subset \textrm{Gr}^+H$,
	\item $W_{k+1} \subset W_k \forall k$, 
	\item $zW_k = W_{k+n}$.
	\item For every flag $\{W_k\}_{k\in \mathbb{Z}}$ the map $\nu:(\{W_k\}_{k\in \mathbb{Z}})\longrightarrow \mathbb{Z}$ mapping every subspace $W_k$ to its virtual dimension is a surjection onto $\mathcal{K}$.
\end{enumerate}
\end{definition}
\end{samepage}

A flag is full iff $\mathcal{K}=\mathbb{Z}$. It is trivial iff $m_{\mathcal{K}}=1$. Trivial flags are in bijection with elements of $Gr^{n,+}(H)$ under the identification $Gr^{n,+}(H)\ni W_0 \leftrightarrow \{z^k W_0\}_{k\in \mathbb{Z}}\in Fl_{\nu(W_0)+n\mathbb{Z}}^{n,+}$.
\noindent A completely symmetric theory can be developed for $H^-$. 

\begin{theorem}
 The flag complex of all periodic flags is a universal geometric twin building $\mathfrak{B}$. Positive flags correspond to simplices in $\mathfrak{B}^+$, negative ones to simplices in $ \mathfrak{B}^-$
\end{theorem}

\begin{proof} This is the main result of chapter $6$ of  \cite{Freyn09}. \end{proof}

\noindent We give some further hints:
One starts by defining apartments via frames, i.e.\ a set of $1$\ndash dimensional subspaces, that span $V$. Then one proves two lemmas showing that those apartments satisfy the axioms in the simplicial complex definition of a building. Using those, one shows, that each connected component is a building. Then one checks using twin apartments, that $\mathfrak{B}=\mathfrak{B}^+ \cup \mathfrak{B}^-$ is a universal twin building (cf.~\cite{AbramenkoRonan98} for the concept of twin apartments).

 \begin{definition}[frames and affine Weyl group]

A frame is a sequence of subspaces $\{U_k\}_{k\in \mathbb{Z}}\subset H^n$ such that $U_{k+n}=z U_k$ and $H^n=\bigcup U_k$. We call a permutation $\pi: \{U_k\}_{k \in\mathbb{Z}}\longrightarrow \{U_k\}_{k \in\mathbb{Z}}$ admissible if $\pi(U_{k+n})= \pi(U_k)+n$.
The affine Weyl group $W_{\textrm{aff}}$ is defined to be the group of admissible permutations of $\{U_k\}_{k \in\mathbb{Z}}$.

\end{definition}

\begin{definition}[apartment]

Let $\{W_k\}$ be a full flag, $\{\leq \hspace{-3pt}W_k\}$ the set of all partial flags in $\{W_k\}$, hence in simplicial complex language, the boundary of $\{W_k\}$ and $\{U_k\}_{k \in\mathbb{Z}}$ a frame. The apartment $\mathcal{A}(\{W_k\}, U_k)$ consists of all flags that are transformations of flags in $\{\leq W_k\}$ by admissible permutations of $\{ \{U_k\}_{k \in\mathbb{Z}}\}$.
\end{definition}

\begin{lemma}
\label{flagapartment}
For a pair of two flags $\{W_k\}$ and $\{W'_{k'}\}$ there exists an apartment containing both of them iff they are compatible in the sense that for all elements $W_k \in \{W_k\}$  there are elements $W'_{i'},W'_{l'}\in \{W'_{k'}\}$ such that $W'_{i'} \subset W_k \subset W'_{l'}$ and vice versa.
Compatibility defines an equivalence relation on the space of flags.
\end{lemma}

\begin{proof}
As we have seen the complex of all flags is a chamber complex. Hence without loss of generality we can assume that $\{W_k\}$ and $\{W_k'\}$ are two maximal compatible flags. For each $k \in \mathbb{Z}$, we define the set $\pi(k):= 
\{m|\exists v \in (W_m\backslash W_{m+1})\cap (W_k'\backslash W_{k+1}')\}$. We have to show that $|\pi(k)|=1$ for all $k$.
So for $i\in \{0,\dots, n-1\}$ we choose vectors $v_i \subset \pi(i)$ and put $U_i = \textrm{span}\langle v_i\rangle$. Furthermore for $i'=ln+i$ set $U_{i'}=U_{ln+i}= z^l U_{i}$.

The proof now consists of several steps:
\begin{itemize}
\item[-]  $\{U_k\}$ is a periodic frame.  As the flags $\{W_k\}$ and $\{W_k'\}$ are periodic, $v\in (W_m\backslash W_{m+1})\cap (W_k'\backslash W_{k+1}')$ is equivalent to $z^l v \in (W_{m+ln}\backslash W_{m+1+ln})\cap (W_{k+ln}'\backslash W_{k+1+ln}')$. 
\item[-]  $W_k= W_{k+1}\oplus U_k$, $W_m= W_{m+1}\oplus U_k$ for all $k$ and $m\subset \pi(k)$. So the apartment associated to $\{U_k\}$ contains $\{W_k\}$ and $\{W_k'\}$. 
\item[-]  $\pi(k+n)=\pi(k)+n$ follows from the periodicity of $\{W_k\}$ and $\{W_k'\}$.
\item[-]  So we are left with showing that $\pi$ is a permutation, that is $|\pi(k)|=1\hspace{1ex}\forall k$. 
The compatibility condition gives 
$$z^{l+1}\{W_k\}=\{W_{k+(l+1)n}\}\subset z\{W_k'\}=\{W_{k+n}'\}\subset\{W_{k+1}'\} \subset \{W_k'\}\subset z^{-l}\{W_k\}\,. $$

So $W'_k\backslash W'_{k+1} \subset W_{k-ln}$. This shows that there are numbers $m$ such that the set $(W_m\cap W_k'\backslash W_{k+1}')$ is nonempty. 
On the other hand $W_{k+(l+1)n} \subset W'_{k+1}$ shows that the set of those $m$ is bounded from above. 
So there is for every $k$ a maximal $m$ such that $(W_m\cap (W_k'\backslash W_{k+1}')$ is nonempty. 
But then $(W_{m+1}\cap (W_k'\backslash W_{k+1}'))$ is empty.
So $(W_m \backslash W_{m+1}')\cap (W_k'\backslash W_{k+1}')$ is nonempty. So $\pi(k)$ is nonempty for all $k$. Symmetrically also $\pi^{-1}(m)$ is nonempty for all $m$. We use now the periodicity condition: $\pi(n+k)=\pi(k)+n$.  As each set $\pi(k)$ is nonempty, for every $l\in \{0, \dots, n-1\}$ there is $k$, such that $\pi(k)=l(\textrm{mod } n)$. 

This means that $\pi(k+n\mathbb Z)=l+n\mathbb{Z}$ as for each $l$, $l+n \mathbb{Z}$ is in the image.  The pigeon hole principle asserts that $|\pi(k)|=1 (\textrm{mod n})$.
 Then the periodicity shows that $|\pi(k)|=1$. Hence $\pi$ is a permutation and thus an element of $W_{\textrm{aff}}$. 
\qedhere

\end{itemize}
\end{proof}

\noindent Furthermore by direct calculation one finds:

\begin{lemma} [apartments are isomorphic]
\label{axiombuilding2foralgebraicflagbuildings}

For every pair of apartments $\mathcal{A}$ and $\mathcal{A}'$ there is an isomorphism $\varphi : \mathcal{A}\longrightarrow \mathcal{A}'$. If $\{W_k\}, \{W'_{k'}\} \subset \mathcal{A}\cap \mathcal{A}'$ such that $\{W_k\}$ is full, one can choose $\varphi$ in a way that it fixes $\{W_k\}$ and $\{W'_{k'}\}$ pointwise.

\end{lemma}

\noindent Those two lemmas yield the theorem:

\begin{theorem}[Tits building]
\label{titsbuilding}
The simplicial complex associated to each equivalence class of compatible flags is an affine Tits building of type $\widetilde{A}_{n-1}$.
\end{theorem}

\begin{cor}
The simplicial complex of positive (partial) flags in $Gr_0^{n,+}$ is an algebraic affine Tits building.
\end{cor}

For the special case of algebraic subspaces i.e.\ elements of the algebraic Grassmannian this construction coincides with the well-known lattice description~\cite{Mitchell88}, \cite{Garrett97}, \cite{AbramenkoNebe02}, \cite{Kramer02}. Similar constructions are possible for the other classical groups~\cite{Freyn09} for a sketch.

\section{Conclusion and outlook}
\label{section:conclusion_and_outlook}

As we have shown Kac-Moody geometry has reached a mature state. The classical objects whose symmetries are described by semisimple Lie groups --- symmetric spaces, polar actions, isoparametric submanifolds, buildings --- have infinite dimensional counterparts whose symmetries are defined by analytic affine Kac-Moody groups. Furthermore important results known from the finite dimensional theory that describe the connections between those objects, generalize to the infinite dimensional setting. The crucial point of the theory is a control of the functional analytic framework that permits a generalization of the algebraic operations. Hence while differential geometric in spirit, the field incorporates various concepts of infinite dimensional Lie theory, functional analysis and algebra.

There are still many important open problems.

\begin{enumerate}
 \item {\bf Mostow rigidity:} Study quotients of Kac-Moody symmetric spaces of the noncompact type. We
  conjecture the existence of a Mostow-type theorem for Kac-Moody
  symmetric spaces of the noncompact type, if the rank $r$ of each irreducible factor satisfies $r\geq 4$; 
  in the finite dimensional situation the main ingredients for the
  proof of Mostow rigidity are the spherical buildings which are
  associated to the universal covers $\widetilde{M}$ and
  $\widetilde{M}'$ of two homotopy equivalent locally symmetric spaces
  $M=\widetilde{M}/\Gamma$ and $M'=\widetilde{M'}/\Gamma'$ (suppose the rank $r$ of each de Rham factor satisfies
  $\textrm{rank}(M)\geq 2$). To prove Mostow rigidity one has to show that a homotopy
  equivalence of the quotients lifts to a quasi isometry of the universal covers and
  induces a building isomorphism. This step is done via a description of the boundary at infinity. By rigidity results of Jacques Tits
  this building isomorphism is known to introduce a group
  isomorphism which in turn leads to an isometry of the quotients.

  Hence to prove a generalization of Mostow rigidity to quotients of
  Kac-Moody symmetric spaces of the noncompact type along these lines, one needs a description of the boundary of Kac-Moody symmetric spaces of the non compact type. In view of their Lorentz structure those spaces are not CAT(0). Consequently a direct generalization of the finite dimensional ideas to construct a boundary seems difficult. Nevertheless as for each point in the symmetric space a boundary can be defined, a complete definition using the action of the Kac-Moody groups seems within reach. We note that a generalization or adaption of the methods developed 
in \cite{Caprace09} and in \cite{Gramlich09} might lead to a proof of Mostow rigidity based on local\ndash global methods. By work of Andreas Mars this is possible for algebraic Kac-Moody groups in case of rings with sufficiently many units \cite{Mars10}.
\item{\bf Holonomy:} Study the holonomy of infinite dimensional Lorentz manifolds. More precisely: Is there a generalization of the Berger holonomy theorem to Kac-Moody symmetric spaces? In the finite dimensional Riemannian case, Berger's holonomy theorem tells us the following: Let $M$ be a simply connected manifold with irreducible holonomy. Then either the holonomy acts transitively on the tangent sphere or $M$ is  symmetric.  Kac-Moody symmetric spaces are Lorentz manifolds and have a distinguished, parallel lightlike vector field, corresponding to $c$. In general relativity similar finite dimensional objects are known as Brinkmann waves~\cite{Brinkmann25}. Hence it seems that Kac-Moody symmetric spaces should be understood as a kind of infinite dimensional Brinkmann wave. The conjectured holonomy theorem states then that the holonomy of an infinite dimensional Brinkmann wave is either transitive on the horosphere around $c$ or the space is a Kac-Moody symmetric space.
\item{\bf Characterization:} Is there some intrinsic characterization of Kac-Moody symmetric spaces? Probably the existence of sufficiently many finite dimensional flats and the closely related Fredholm property of the induced polar actions will be part of this characterization. From an algebraic point of view, it would be very interesting, to investigate if there are symmetric spaces associated to Lorentzian Kac-Moody algebras - nevertheless, the absence of good explicit realizations seems to be a serious impediment. From a functional analytic point of view, a generalization to other classes of maps, i.e.\ holomorphic maps on Riemann surfaces or symplectic maps would be interesting. From the finite dimensional blueprint the development of an infinite dimensional version of the theory of non-reductive pseudo-Riemannian symmetric spaces as developed in the finite dimensional case by Ines Kath and Martin Olbricht in \cite{Kath06} seems promising.  In view of the characterization of Kac-Moody symmetric spaces as infinite dimensional Brinkmann waves, we would like to understand, how the existence of this special direction is related to other structure properties. It is hence natural, to ask, if there is a class of infinite dimensional Lorentz symmetric spaces without such a distinguished direction? 
\item{\bf \glqq Semi\ndash Riemannian\grqq\ Kac-Moody symmetric spaces:} We argued that the Kac-Moody symmetric spaces, we developed correspond to Riemannian symmetric spaces.  Marcel Berger studied reductive symmetric spaces of arbitrary index~\cite{Berger57}. The Kac-Moody analogue should exist but is not constructed nor is a classification available. Remark that a classification can be given without a construction of the spaces just by considering involutions of Kac-Moody algebras~\cite{Heintze09}.  
 \item{\bf Kac-Moody symmetric spaces as Moduli space:} Recent results
  by Shimpey Kobayashi and Josef Dorfmeister show that the moduli
  spaces of different classes of integrable surfaces can be understood
  as real forms of loop groups of $\mathfrak{sl}(2,\mathbb{C})$~\cite{Kobayashi09}. We conjecture that Kac-Moody symmetric spaces
  can be interpreted as Moduli spaces of special classes of submanifolds in
  more general situations.

\end{enumerate}

\end{document}